\documentclass[11pt, reqno]{amsart}

\usepackage{graphicx}
\usepackage{amsmath,amssymb,mathtools,amsopn,amsfonts,amsthm}
\usepackage{bbm}
\usepackage{subfigure}
\usepackage{eepic}
\usepackage{tikz}
\usepackage{a4wide}
\usepackage[utf8]{inputenc}
\usepackage{epsfig}
\usepackage{latexsym}
\usepackage{enumerate, enumitem}
\usepackage[UKenglish]{babel}
\usepackage{verbatim}
\usepackage[initials]{amsrefs}
\usepackage[bookmarks]{hyperref}

\newtheorem{thm}{Theorem}[section]
\newtheorem{lma}[thm]{Lemma}

\newtheorem{cor}[thm]{Corollary}
\newtheorem{prop}[thm]{Proposition}
\theoremstyle{definition}
\newtheorem{defn}[thm]{Definition}

\newtheorem{rem}[thm]{Remark}

\newcommand{\R}{\mathbb{R}}

\newcommand{\N}{\mathbb{N}}

\newcommand{\Z}{\mathbb{Z}}
\newcommand{\C}{\mathbb{C}}

\newcommand{\SL}{\mathrm{SL}(2,\mathbb{R})}
\renewcommand{\H}{\mathbb{H}}
\renewcommand{\P}{\mathbb{P}}
\providecommand{\norm}[1]{\lVert#1\rVert}

\newcommand{\I}{\mathcal{I}}

\newcommand{\A}{\mathcal{A}}
\renewcommand{\i}{\mathtt{i}}
\renewcommand{\j}{\mathtt{j}}

\newcommand{\hd}{\dim_\textup{H}}

\newcommand{\att}{\Lambda^+}
\newcommand{\rep}{\Lambda^-}
\renewcommand{\epsilon}{\varepsilon}
\renewcommand{\geq}{\geqslant}
\renewcommand{\leq}{\leqslant}

\title{The Hausdorff dimension of self-projective sets}

\author{Argyrios Christodoulou} \address{Department of Mathematics, University of Surrey, Guildford, GU2 7XH}
\email{a.christodoulou@surrey.ac.uk}

\author{Natalia Jurga} \address{Mathematical Institute, University of St Andrews, Scotland, KY16 9SS}
\email{naj1@st-andrews.ac.uk}

\begin{document}

 \maketitle

\begin{abstract}
Given a finite set $\A \subseteq \SL$ we study the dimension of the attractor $K_\A$ of the iterated function system induced by the projective action of $\A$. In particular, we generalise a recent result of Solomyak and Takahashi by showing that the Hausdorff dimension of $K_\A$ is given by the minimum of 1 and the critical exponent, under the assumption that $\A$ satisfies certain discreteness conditions and a Diophantine property. Our approach combines techniques from the theories of iterated function systems and M\"obius semigroups, and allows us to discuss the continuity of the Hausdorff dimension, as well as the dimension of the support of the Furstenberg measure.
\end{abstract}

\section{Introduction}

This article studies the action of $\SL$ matrices on one-dimensional real projective space $\R\P^1$ induced by the linear action of $\SL$ on $\R^2$. We identify $\R\P^1$ with the interval $(0,\pi]$, with the endpoints identified, and for a matrix $A\in\SL$, we denote the induced projective map by $\phi_A : \R\P^1 \to \R\P^1$ (for more information we refer to \cite[Section 2]{hochman_sol}). 
Given a finite set $\A \subseteq \SL$ we call the collection of maps $\Phi_\A\vcentcolon=\{\phi_A\}_{A\in\A}$ a projective iterated function system (projective IFS). Iterated function systems on real projective space have been previously studied in \cite{bv,deleo1, deleo2, solomyak} and on complex projective space in \cite{vince}. 

Recall that a non-identity matrix $A\in\SL$ is called \emph{elliptic} if $\lvert\mathrm{tr}(A)\rvert<2$, \emph{parabolic} if $\lvert\mathrm{tr}(A)\rvert=2$ and \emph{hyperbolic} if $\lvert\mathrm{tr}(A)\rvert>2$. The corresponding projective map $\phi_A$ will be called elliptic (resp. parabolic or hyperbolic) if the matrix $A$ is elliptic (resp. parabolic or hyperbolic). A parabolic map $\phi_A$ has a unique fixed point $x_0$ in $\R\P^1$, where $\phi_A'(x_0)=1$. A hyperbolic map has two fixed points; an attracting $a$ where $\phi_A'(a)<1$, and a repelling $r$ where $\phi_A'(r)>1$. On the other hand, elliptic maps are non-trivial rotations of $\R\P^1$ and have no fixed points. We denote by $\langle \Phi_\A \rangle$ the semigroup generated by $\Phi_\A$.

\begin{defn}\label{attdef}
For a set $\A\subseteq \SL$ we define $K_\A\subseteq \R\P^1$ to be the smallest closed set containing all attracting fixed points of hyperbolic maps in $\langle \Phi_\A \rangle$ and all unique fixed points of parabolic maps in $\langle \Phi_\A \rangle$.  The set $K_\A$, if non-empty, is called the \emph{attractor} of $\Phi_\A$.
\end{defn}
In principle, $K_\A$ may be empty or the whole of $\R\P^1$. The attractor will be the main object of study in this paper and our main result is a generalisation of a recent result of Solomyak and Takahashi  \cite[Theorem~1.7]{solomyak} concerning its Hausdorff dimension. Before stating our main result, we introduce the two key assumptions that we will need to make on our set $\A$.

If $\A=\{A_i\}_{i \in \I} \subseteq \SL$, we denote by $\A^n$ all products of $n$ matrices from $\A$ and by $\A^*\vcentcolon=\bigcup_{n=1}^\infty \A^n$ the semigroup generated by $\A$. 

\begin{defn}
A set $\A\subseteq \SL$ is called \emph{semidiscrete} if $\mathrm{Id}\notin\overline{\A^*}$, where the closure is taken in $\SL$.
\end{defn}
The term ``semidiscrete" was introduced in \cite{jasho}, where Jacques and Short used a slightly weaker definition (our semidiscrete sets are called ``semidiscrete and inverse-free" in the language of \cite{jasho}). Note that semidiscrete sets can contain parabolic and hyperbolic matrices, but not elliptic matrices. We are later going to see that the semidiscrete property of $\A$ is enough to guarantee that $\Phi_\A$ exhibits some contractive properties on $\R\P^1$.

Next we introduce the ``Diophantine property" which appears in \cite[Definition 1.1]{solomyak} in the following form.
\begin{defn}
Let $\A=\{A_i\}_{i \in \I} \subseteq \SL$ be finite, and let $d$ be a left-invariant Riemannian metric in $\SL$. We say that $\A$ is \emph{Diophantine} if there exists $c>0$ such that for all $n \in \N$, if $\i, \j \in \I^n$ and $\i \neq \j$, we have
\[
d(A_\i,A_\j)>c^n.
\]
\end{defn}
Note that the Diophantine property is independent of the choice of left-invariant Riemannian metric $d$, see \cite[\S 2.3]{hochman_sol}. It is clear that if $\A$ is Diophantine then it generates a free semigroup. In \cite{hochman_sol,solomyak} the authors also consider a weaker version of this definition with the freeness of $\A^*$ removed, but we do not require this more general form.

Given a finite or countable set $\A \subseteq \SL$, define its \emph{zeta function} $\zeta_\A:[0,\infty) \to \R \cup \{\infty\}$ by
\begin{equation}
\label{zeta}
\zeta_\A(s)\vcentcolon= \sum_{n=1}^\infty \sum_{A \in \A^n} \norm{A}^{-2s}
\end{equation}
and its \emph{critical exponent}
\begin{equation}
\label{crit}
\delta_\A\vcentcolon= \inf\{s>0 \; : \; \zeta_\A(s)<\infty\}.
\end{equation}
If $\zeta_\A(s)$ diverges for all $s \geq 0$, we define $\delta_\A=\infty$. 

We are ready to state the main result of this paper.
\begin{thm}\label{MAIN}
Let $\A$ be a finite subset of $\SL$ that is Diophantine and semidiscrete. If $K_\A$  is not a singleton, then
\begin{equation} \label{dimformula}
\hd(K_\A)=\min\{1,\delta_\A\}.
\end{equation}
\end{thm}

A first observation about Theorem~\ref{MAIN} is that it can be applied to projective IFS containing parabolic maps. In order to properly describe the contribution of Theorem \ref{MAIN}, however, we must first review the literature which is most relevant to our problem.
To this end, we introduce the notion of uniform hyperbolicity following \cite{bv, deleo1, solomyak}.
\begin{defn} \label{uh}
A finite set $\A\subseteq \SL$ is called \emph{uniformly hyperbolic} if there exist real numbers $\lambda>1$ and $c>0$, such that for every $n\in\N$,
\[
\norm{A}\geq c\lambda^n, \quad \text{for all} \quad A\in\A^n.
\]
\end{defn}
A concept similar to uniform hyperbolicity for $\mathrm{GL}(d,\R)$-cocycles was studied by Bochi and Gourmelon in \cite{bogo}. Note that if $\A$ is uniformly hyperbolic, then all the matrices in the semigroup $\A^*$ have to be hyperbolic.

As we mentioned earlier, Theorem~\ref{MAIN} is a generalisation of the following theorem due to Solomyak and Takahashi \cite[Theorem~1.7]{solomyak}.
\begin{thm} \label{sol_thm}
Let $\A \subseteq \SL$ be a finite, uniformly hyperbolic set of matrices which is Diophantine. If $K_\A$ is not a singleton, then
\[
\hd K_\A=\min\{1,\delta_\A\}.
\]
\end{thm}
In the case that $\Phi_\A$ (when restricted to an appropriate open subset of $\R\P^1$, such as a multicone, see \S 2) additionally satisfies the open set condition, Theorem \ref{sol_thm} is due to De Leo \cite[Theorem 4]{deleo1}. We also note that a different definition was used for the attractor in \cite{solomyak} (see Section \ref{s2}), but in Theorem \ref{dense} we show that it is equivalent to Definition \ref{attdef}.

To clarify how our result generalises Theorem~\ref{sol_thm}, let us consider the parameter space $\SL^N$ for some positive integer $N$. Also, let $\mathcal{H}$ denote the set of all $N$-tuples that are uniformly hyperbolic (as a slight abuse of notation we use $N$-tuples and subsets of $N$ elements interchangeably). The locus $\mathcal{H}$ was thoroughly investigated by Avila, Bochi and Yoccoz \cite{aby}, as well as Yoccoz \cite{yoccoz}, where they proved various topological properties of $\mathcal{H}$ and raised several questions, many of which remain still open.

Similarly to $\mathcal{H}$, we let $\mathcal{S}$ denote the set of all $N$-tuples that are semidiscrete. It is easy to see that if $\A\subseteq \SL$ is finite and uniformly hyperbolic then it is also semidiscrete. So we have the inclusion $\mathcal{H}\subseteq \mathcal{S}$. More importantly, most $N$-tuples on the boundary of $\mathcal{H}$, also lie in $\mathcal{S}$ \cite[Theorem~5.17]{thesis}. The only exceptions for this statement are certain $N$-tuples that lie on the boundary of specific components, called principal components, but in general these are easy to handle (we refer to \cite[Section 2]{aby} as well as \cite[Section~5.3]{thesis}, for more information). Thus, the parameter space $\mathcal{S}$ is significantly larger than $\mathcal{H}$ and contains boundary $N$-tuples that are normally difficult to work with through the lens of uniform hyperbolicity. As an example, we note the existence of semidiscrete $N$-tuples that lie on the boundary of $\mathcal{H}$ but not on the boundary of any connected component of $\mathcal{H}$ \cite[Example~6.5]{thesis}. In addition, the tuple $\A$ in \cite[Example~6.5]{thesis} is such that the semigroup $\A^*$ contains only hyperbolic matrices, indicating that the differences between the semidiscrete and uniformly hyperbolic properties are more subtle than merely the inclusion of parabolic matrices in the former. We shall return to this example at the end of Section \ref{s2}.

Thus, Theorem \ref{MAIN} allows us to obtain a dimension formula for the attractor of Diophantine tuples on the boundary of most components of $\mathcal{H}$, as well as any tuples that are points of accumulation of components.

Let us now discuss the common ground between Theorems \ref{MAIN} and \ref{sol_thm}, the Diophantine property. It is known that when the matrices in $\A$ have algebraic entries and generate a free semigroup, $\A$ is Diophantine \cite[Proposition 4.3]{hochman_sol}. Solomyak and Takahashi \cite[Theorem 1.2]{solomyak} proved that if $\SL$ is treated as a subset of $\R^{4|\I|}$, then Lebesgue almost all choices of sets of positive matrices $\A\subset \SL$ are Diophantine. In general however it is not known whether the Diophantine property is generic amongst tuples of matrices from $\SL$, and inquiries related to this are a topic of contemporary research, see \cite{solomyak} and references therein.

The Diophantine property is also closely related to a separation condition which has recently received attention in the dimension theory of IFS on $\R^d$, called the ``exponential separation condition" \cite{hochman-annals,rap}. The exponential separation condition states that images of points in $\R^d$ under compositions of fixed length do not get super-exponentially close in the topology of $\R^d$ (see \cite[Definition~1.9]{solomyak} for a more precise statement when $d=1$). The notion of exponential separation was introduced by Hochman  who proved that if an IFS composed of similarity mappings satisfies the exponential separation condition, then the Hausdorff dimension of its attractor is given by a natural formula \cite[Corollary 1.2]{hochman-annals}. This marked significant progress in
a long standing problem on overlapping self-similar sets.

Given a finite set $\A \subseteq \SL$ we can consider the IFS of M\"obius transformations that it induces on $\R$. For example, if $\A$ are chosen in such a way that it induces an IFS of similarity mappings, then the natural formula from \cite[Corollary 1.2]{hochman-annals} corresponds to $\min\{1,\delta_\A\}$ for this choice of $\A$.   Solomyak and Takahashi proved that if an IFS of M\"obius transformations on $\R$, induced by a set of matrices $\A$, satisfies the exponential separation condition then $\A$ satisfies the Diophantine property \cite[Proposition 2.4]{solomyak}. Thus Theorem \ref{sol_thm}, as well as our result Theorem~\ref{MAIN}, can be interpreted as extensions of Hochman's result.

Finally, we present two corollaries that follow from Theorem~\ref{MAIN}. First, a corollary concerning the dimension of the support of stationary measures associated to $\A$. For this, we will require the notions of irreducibility and strong irreducibility of a set $\A \subseteq \SL$.
\begin{defn}
A set $\A\subseteq \SL$ will be called \emph{strongly irreducible}, if the set of maps $\Phi_\A$ do not preserve any finite subset of $\R\P^1$. A set $\A \subseteq \SL$ will be called \emph{irreducible} if the set of maps $\Phi_\A$ do not have a common fixed point in $\R\P^1$. If $\A$ is not irreducible, then we say that it is \emph{reducible}.
\end{defn}

Given a finite set $\A=\{A_i\}_{i \in \I} \subseteq \SL$ and non-degenerate probability vector $(p_i)_{i \in \I}$, we can consider the probability measure $\mu$ on $\SL$ which is supported on $\A$:
\[
\mu=\sum_{i \in \I} p_i \mathbf{1}_{A_i}.
\]

If $\A$ is strongly irreducible and generates an unbounded semigroup, then there exists a unique probability measure $\nu$ on $\R\P^1$ with the property that
\[
\nu=\sum_{i \in \I} p_i \phi_{A_i}^*\nu,
\]
where $\phi_{A_i}^*\nu$ denotes the pushforward of $\nu$ under the map $\phi_{A_i}$, see \cite{fu}. We call $\nu$ the \emph{Furstenberg measure} (or the stationary measure) associated to $\A$ and $(p_i)_{i \in \I}$. We note that if $\A \subseteq \SL$ is semidiscrete, it is irreducible if and only if it is strongly irreducible (Lemma \ref{strong}), therefore given an irreducible, semidiscrete subset $\A \subseteq \SL$ and a non-degenerate probability vector $(p_i)_{i \in \I}$ we can refer to the (unique) Furstenberg measure without any ambiguity.

The dimension theoretic properties of $\nu$, such as its exact dimensionality, were investigated first by Ledrappier \cite{ledrappier} followed by Hochman and Solomyak \cite{hochman_sol} who obtained a more practical formula for its dimension, which provided the main tool behind the proof of Theorem \ref{sol_thm}. On the other hand, Theorems \ref{MAIN} and \ref{sol_thm} allow one to study the dimension of the \emph{support} of the Furstenberg measure. In \cite[Corollary 1.12]{solomyak}, Solomyak and Takahashi showed that if $\A$ is Diophantine, uniformly hyperbolic and irreducible, then $\hd \mathrm{supp} \nu=\min\{1,\delta_\A\}$. Theorem \ref{MAIN} allows us to extend their result in the following way.

\begin{cor}\label{furst}
Suppose $\A \subseteq \SL$ is finite, Diophantine, semidiscrete and irreducible. Let $\nu$ be the Furstenberg measure for a non-degenerate probability vector $(p_i)_{i \in \I}$. Then
\[
\hd \mathrm{supp} \nu= \min\{1, \delta_\A\}.
\]
\end{cor}

Our last result is a continuity property for the Hausdorff dimension on Diophantine subsets of $\mathcal{S}$.

\begin{thm}\label{cty}
Let $(\A_n)$ be a sequence of Diophantine and semidiscrete subsets of $\SL$. Assume that $\A_n$ converges to $\A$ in the Hausdorff metric, where $\A\subset\SL$ is semidiscrete, Diophantine and $K_\A$ is not a singleton. Then $\hd K_{\A_n} \to \hd K_\A$.
\end{thm}

We remark that there are also many points of discontinuity of the Hausdorff dimension in $\SL$ and as such, the assumptions on the limit point $\A$ in Theorem \ref{cty} cannot be significantly weakened, see Remark \ref{discont}. The proof of Theorem \ref{cty} is largely based on two results: an identification of the critical exponent of the zeta function of $\A$ with the minimal root of its pressure function (see \eqref{pressure}, in Section~\ref{irred}, for the definition), and a continuity result for the critical exponent, Theorem~\ref{cty_str}. These results, however, require further setup before they can be stated properly.\\

\noindent\textbf{Structure of the paper.} In the second section we present a geometric approach to how the uniformly hyperbolic and semidiscrete properties affect the contraction of the induced IFS. Section~\ref{examples} is dedicated to applying Theorem~\ref{MAIN} to three classes of projective IFS that are not uniformly hyperbolic. The main properties of the attractor will be proved in Section~\ref{attractor}, where we also settle Corollary~\ref{furst}.

The proof of our main result, Theorem~\ref{MAIN}, is carried out in two steps. First, we settle the theorem for reducible sets of matrices in Section~\ref{eg}, before moving to the more general, irreducible, case in Section~\ref{irred}. Finally, Section~\ref{cty_sect} contains the proof of Theorem~\ref{cty}.

\section{Contractive properties for projective IFS} \label{s2}

There is a strong connection between the uniform hyperbolicity of a finite set of matrices and the contractive properties of the projective IFS that it induces. Before we describe this connection, in Theorem~\ref{abybv} to follow, we make the following definition.

\begin{defn}
A subset of $\R\P^1$ will be called a \emph{multicone} if it is a finite union of open intervals with disjoint closures.
\end{defn}

For a set $U\subseteq \R\P^1$, define $\Phi_\A(U)\vcentcolon= \bigcup_{A \in \A} \phi_A(U)$. We say that $\Phi_\A$ maps a subset $U \subseteq \R\P^1$ \emph{compactly inside itself} if $\Phi_\A(\overline{U}) \subsetneq U$ and \emph{strictly inside itself} if $\Phi_\A(U) \subsetneq U$. Also, let $d_\P$ denote the metric induced by identifying $\R\P^1$ with $\R/\pi\Z$ (see \cite[Section 2]{hochman_sol} for a formula for $d_\P$). The next theorem can be obtained by combining the main results from \cite{aby} and \cite{bv}.

\begin{thm}\label{abybv}
Let $\A$ be a finite subset of $\SL$. The following are equivalent:
  \begin{enumerate}
  \item $\A$ is uniformly hyperbolic.
  \item $\Phi_\A$ maps a multicone compactly inside itself.
  \item There is a non-trivial open set $U \subsetneq \R\P^1$ and a metric $d$ in $\R\P^1$, equivalent to $d_\P$, such that $\Phi_\A$ is uniformly contractive on $\overline{U}$ with respect to $d$.
\item There exists a non-empty closed set $F \subsetneq \R\P^1$, such that $\Phi_\A(F)=F$ and $\Phi_{\A^n}(B)$ converges to $F$ as $n\to\infty$, in the Hausdorff metric in $\R\P^1$, for all non-empty closed sets $B$ in an open neighbourhood of $F$.
  \end{enumerate}
  \end{thm}

So, a uniformly hyperbolic set $\A$ necessarily induces a uniformly contractive iterated function system on any multicone which is compactly contracted by $\Phi_\A$. Barnsley and Vince \cite{bv} defined the attractor of a projective IFS $\Phi_\A$ as the set $F$ which satisfies \textit{(4)} above, and showed that it exists if and only if $\A$ is uniformly hyperbolic, in which case it is unique. The same definition of an attractor was used in \cite{solomyak}. In Section~\ref{attractor}, we show that our definition of an attractor is in fact a generalisation of the definition from \cite{bv}, in the sense that the two coincide whenever $\A$ is uniformly hyperbolic. Therefore, we freely use the term ``attractor" without fear of confusion between the two definitions.

On the other hand, semidiscrete subsets of $\SL$ do not generally induce a uniformly contractive iterated function system on any subset of $\R\P^1$ (see, for example, Lemmas \ref{para-rep} and \ref{att-rep}). Nevertheless, an analogue of Theorem~\ref{abybv} for semidiscrete subsets of $\SL$ was proved by Jacques and Short \cite[Theorem~7.1]{jasho}, and can be stated as follows.

\begin{thm}\label{sd}
Let $\A$ be a finite subset of $\SL$. If $\Phi_\A$ maps a closed subset of $\R\P^1$ strictly inside itself, then $\A$ is semidiscrete. Conversely, if $\A$ is semidiscrete, then there exists a non-trivial closed subset $C\subsetneq \R\P^1$, such that $\Phi_\A(C)\subseteq C$.
\end{thm}

The second part of Theorem~\ref{sd} cannot be improved since there are examples where there exists a unique interval fixed by $\Phi_\A$ and yet $\A$ is semidiscrete (see \cite[Section~9]{jasho}). Theorem~\ref{sd} indicates that even though a projective IFS induced by a semidiscrete set $\A$ may not contain topological contractions, its action on $\R\P^1$ still exhibits some contractive properties.

Comparing Theorems \ref{abybv} and \ref{sd}, we obtain a geometric interpretation of the fact that uniformly hyperbolic subsets of $\SL$ are semidiscrete. The simplest example of a set $\A$ that is semidiscrete and not uniformly hyperbolic, is one where $\Phi_\A$ maps a multicone strictly, but not \emph{compactly} inside itself. One can find many examples of such systems by considering subsets $\A  \subseteq \SL$ which lie on the boundaries of connected components of $\mathcal{H}$. It is important to note however that not all examples are of this type. For instance, \cite[Example~6.5]{thesis} mentioned in the introduction demonstrates a semidiscrete set $\A$ where $\Phi_\A$ does not map any multicone inside itself, but instead maps an infinitely connected closed subset of $\R\P^1$ inside itself (an infinitely connected version of a multicone).

\section{Examples}\label{examples}
Before we describe the proofs of our results, we apply our main theorem, Theorem~\ref{MAIN}, to three templates of projective iterated functions systems that are not induced by uniformly hyperbolic sets of matrices. These offer some insight on the restrictions posed by uniform hyperbolicity and the need for a more general approach on projective IFS.\\

\noindent\textbf{Systems with parabolics.}  Suppose that $\A$ is a finite, Diophantine and semidiscrete subset of $\SL$. Assume the attractor $K_\A$ of $\A$ is not a singleton, and that there exists a parabolic matrix $P$ in $\A^*$. The existence of $P$ forces the critical exponent of the zeta function of $\A$ to be large, in some sense which has a similar effect on the Hausdorff dimension of $K_\A$, due to Theorem~\ref{MAIN}.

To be more specific, note that since $P$ is parabolic $\norm{P^n}=\mathrm{O}(n)$. Thus
\[
\zeta_\A(s)\geq \sum_{n=1}^\infty \norm{P^n}^{-2s}= C^{-2s}\sum_{n=1}^\infty n^{-2s},
\]
for some constant $C>0$, and the sum on the right hand side diverges for all $s\leq\tfrac{1}{2}$. So $\delta_\A\geq \tfrac{1}{2}$, which in turn implies that $\hd K_\A\geq \tfrac{1}{2}$.\\

\noindent\textbf{Non-discrete systems.} A semigroup $S\subset\SL$ is called \emph{discrete} if it is a discrete subset in the topology of  $\SL$. It is clear that uniformly hyperbolic sets generate discrete semigroups. The same, however, does not hold for semidiscrete sets (see \cite[Section~3]{jasho} and Lemma~\ref{att-rep}). Even though semidiscrete subsets of $\SL$ that generate non-discrete semigroups exhibit interesting properties, Theorem~\ref{MAIN} renders studying the dimension of their attractor trivial.

Let $\A\subseteq \SL$. We show that $\delta_\A=\infty$ whenever $\A^*$ is not discrete. Since $\A^*$ is not discrete, there exists a sequence $(A_n)\subseteq \A^*$ that converges to some $A\in \SL$. So, there exists some $r>0$ such that $\norm{A_n}<r$, for all $n$. Hence, for all $s>0$
\[
\zeta_\A(s)\geq\sum_{n=1}^\infty \norm{A_n}^{-2s}\geq \sum_{n=1}^\infty\frac{1}{r^{2s}}.
\]
Thus $\zeta_\A(s)$ diverges for all $s>0$, implying that $\delta_\A=\infty$.

If we additionally assume that $\A$ is a finite, Diophantine and semidiscrete subset of $\SL$, and the attractor $K_\A$ is not a singleton, then Theorem~\ref{MAIN} shows that $\hd K_\A=1$.\\

\noindent\textbf{Systems with elliptics.} Our final class of examples is motivated by the work of B\'ar\'any, K\"aenm\"aki and Morris \cite{bakamo}. Suppose that $M$ is a multicone in $\R\P^1$, and let $E, S\subseteq \SL$ be finite non-empty sets that satisfy the following:
\begin{enumerate}
\item $\Phi_S$ maps $M$ strictly inside itself.
\item $E$ contains only elliptic matrices and $\Phi_E$ fixes the multicone $M$.
\end{enumerate}
Notice that Theorem~\ref{sd} implies that $S$ is semidiscrete. Also, since $M$ is a multicone, all elements of $E$ have to be elliptic matrices of finite order. Let $p$ denote the lowest common multiple of the orders of elements in $E$, and choose $B \in E^*$ such that $B$ has order $p$.

Now, consider $\A=S\cup E$. The set $\A$ is neither semidiscrete nor Diophantine, since $\mathrm{Id}\in \A^*$, but its action on projective space is, in some sense, tame due to the fact that $\A$ consists of two ``well-behaved" sets. The authors of \cite{bakamo} are mainly interested in case where $S$ is uniformly hyperbolic, where they prove that it is equivalent to the property that the cocycle induced by $\A$ is almost additive \cite[Corollary~2.5]{bakamo}.

Note that $K_\A$ is non-empty, as the semigroup $\A^*$ contains many hyperbolic or parabolic matrices, but we cannot use Theorem~\ref{MAIN} directly to determine its Hausdorff dimension. Instead, we consider the set 
\begin{align*}
\A_0=\left\{AB^n \; : A \in S, \; 0 \leq n \leq p-1\right\}.
\end{align*}
Since $\Phi_S$ maps $M$ strictly inside itself and $\Phi_E$ fixes $M$, it is obvious that $\Phi_{\A_0}$ maps $M$ strictly inside itself, and thus Theorem~\ref{sd} implies that $\A_0$ is semidiscrete. Furthermore, we can show that $\hd K_\A=\hd K_{\A_0}$.

To see this, let
\begin{align*}
\A_j=\left\{B^jAB^n \; : A \in S, \; 0 \leq n \leq p-1\right\}.
\end{align*}
for each $1 \leq j \leq p-1$. In particular, $K_\A= \bigcup_{j=0}^{p-1} K_{\A_j}$. Let $1 \leq j \leq p-1$. We begin by showing that $K_{\A_0}$ contains a bi-Lipschitz copy of $K_{\A_j}$. To see this, let $C \in \A_j^*$, so that $C=B^jD$ for some $D \in \A_0^*$. Note that either $C, D$ and $DB^j$ are all hyperbolic or they are all parabolic. Assuming they are all hyperbolic, let $a(DB^j)$ denote the attracting fixed point of $\phi_{DB^j}$ and let $a(C)$ denote the attracting fixed point of $\phi_C$ and note that $a(C)=\phi_{B^j}(a(DB^j))$. 
 Similarly, if $C, D$ and $DB^j$ are all parabolic, then letting $p(C)$ denote the parabolic fixed point of $\phi_C$ and letting $p(DB^j)$ denote the parabolic fixed point of $\phi_{DB^j}$ we deduce that $p(C)=\phi_{B^j}(p(DB^j))$. Since $DB^j \in \A_0^*$ we conclude that $\phi_{B^{-j}}(K_{\A_j}) \subseteq K_{\A_0}$, therefore $\hd K_{\A_j} \leq \hd K_{\A_0}$. Similarly we can show that $\hd K_{\A_j} \geq \hd K_{\A_0}$. Since $1 \leq j \leq p-1$ was chosen arbitrarily, $\hd K_\A=\hd K_{\A_0}$.

 So, if we assume that $\A_0$ is Diophantine and that $K_\A$ is not a singleton, then using Theorem~\ref{MAIN} on the set $\A_0$ yields the formula $\hd K_\A=\min\{1,\delta_{\A_0}\}$.

\section{Attractors of projective IFS} \label{attractor}

Let $\A$ be a subset of $\SL$. Recall that the attractor of $\Phi_\A$, where it exists, is defined as the smallest closed set $K_\A\subseteq \R\P^1$  containing all attracting fixed points of hyperbolic maps in $\langle \Phi_\A \rangle$ and all unique fixed points of parabolic maps in $\langle \Phi_\A \rangle$. In this section we prove important properties for $K_\A$ that will be used throughout the rest of the paper. It is important to note that most of our results involving $K_\A$ will not assume that the set $\A$ is finite, and thus can be used even for infinitely-generated IFS.

\subsection{M\"obius transformations and limit sets of semigroups} \label{mobius}
Let us consider the projective space $\R\P^1$ as the extended real line $\overline{\R}=\R\cup\{\infty\}$. As mentioned in the introduction, $\SL$ acts on $\R$ with M\"obius transformations. For a matrix 
\[
A=\begin{pmatrix}
a & b\\
c & d
\end{pmatrix},\]
in $\SL$, we define the map $f_A\colon \overline{\R} \to  \overline{\R}$ with
\[
f_A(x)=\frac{ax+b}{cx+d}.
\]
Through this map, a set of matrices $\A\subseteq \SL$ induces an iterated function system $F_\A=\{f_A\colon A\in\A\}$ on $\overline{\R}$. Also, define $\psi\colon (0,\pi]\to \overline{\R}$ with $\psi(\theta)=\cos \theta\slash\sin \theta$ and note that $\psi$ is a smooth mapping for which the derivatives of $\psi$ and $\psi^{-1}$ are bounded on compact subsets of $(0,\pi)$. Hence, $\psi$ provides us with the following correspondence between the iterated functions systems $\Phi_\A$ and $F_\A$
\begin{equation}\label{psieq}
\psi \circ \phi_A= f_A \circ \psi.
\end{equation}
Observe that for any $A\in\SL$ the map $f_A$ is a linear fractional transformation. Hence, if $\H$ is the upper half-plane in $\C$, the action of $A$ on $\overline{\R}$ can be extended to an action on $\overline{\H}=\H\cup\overline{\R}$. Thus, for any $\A\subseteq \SL$, the set $F_\A$ can be thought of as a set of M\"obius transformations mapping $\H$ conformally onto itself. Also note that the group of all such M\"obius transformations is isomorphic to $\mathrm{PSL}(2,\R)$. We refer the reader to \cite[Section 2]{reyes} for an interesting approach on the connection between $\SL$ and the conformal automorphisms of $\H$. 

Let $\langle F_\A \rangle$ be the semigroup generated by $F_\A$. For a point $z\in\overline{\H}$, we define \emph{the orbit of $z$ under $\langle F_\A \rangle$} to be the set $ F_\A (z)\vcentcolon=\{f_A(z)\colon A\in\A^*\}$. In accordance to the theory of discrete groups, we define the limit set of $F_\A$ as follows.

\begin{defn}\label{deflim}
The \emph{forward limit set} $\att(F_\A)$ of $F_\A$ is defined to be the set $\overline{ F_\A (z_0)}\cap\overline{\R}$, where $z_0\in\mathbb{H}$. Similarly, the \emph{backward limit set} $\rep(F_\A)$ of $F_\A$ is defined to be the forward limit set of $\{{f_A}^{-1}\colon A\in\A\}$.
\end{defn}

The definition of the forward and backward limit sets does not depend on the choice of the point $z_0$, and so we will often chose $z_0$ to be the complex number $i$. Furthermore, observe that $\att(F_\A)$ is forward invariant under $F_\A$, while $\rep(F_\A)$ is backward invariant.

Properties of the limit sets of M\"obius semigroups were studied in \cite{frmast,jasho,hima}. Of particular interest to our analysis is the following result by Fried, Marotta and Stankiewitz \cite[Theorem~2.4, Proposition~2.6 and Remark~2.20]{frmast}.

\begin{thm}\label{rep}
Let $\A\subseteq\SL$ be such that $\A^*$ contains hyperbolic matrices. Then $\rep(F_\A)$ is the smallest closed set in $\overline{\R}$ containing all repelling fixed points of hyperbolic transformations in $\langle F_\A \rangle$. If, in addition, $\rep(F_\A)$ is infinite then it is a perfect set.
\end{thm}

An immediate consequence of Theorem~\ref{rep} is the following characterisation of the forward limit set.

\begin{cor}\label{att}
Let $\A\subseteq\SL$ be such that $\A^*$ contains hyperbolic matrices. Then $\att(F_\A)$ is the smallest closed set in $\overline{\R}$ containing all attracting fixed points of hyperbolic transformations in $\langle F_\A \rangle$. If, in addition, $\att(F_\A)$ is infinite then it is a perfect set.
\end{cor}

\subsection{Attractor versus limit set}
Using \eqref{psieq}, we are going show that the attractor of the IFS $\Phi_\A$ and the forward limit set of $F_\A$ are closely related, as stated in Theorem~\ref{attlimset} to follow. Let us first establish the following lemma.

\begin{lma}\label{reduc}
For any set $\A\subseteq\SL$ the following hold:
\begin{enumerate}
\item If $\#K_\A=1$ then $\A^*$ does not contain any elliptic matrices; whereas
\item if $\#K_\A>1$ then the semigroup $\A^*$ contains a hyperbolic matrix.
\end{enumerate}
\end{lma}

\begin{proof}
We start with the proof of \emph{(1)}. Since $K_\A\neq\emptyset$, the semigroup $\A^*$ has to contain a matrix $A_0$ that is either parabolic or hyperbolic. Assume that there exists an elliptic matrix $B\in\A^*$. If $B$ is of finite order, then the matrices $B^{-1}A_0B\in\A^*$ and $A_0$ are either both parabolic or both hyperbolic and the projective maps $\phi_{A_0}$ and $\phi_{B_1^{-1}A_0B_1}$ have distinct fixed points in $\R\P^1$. Thus $\#K_\A\geq 2$ which is a contradiction. For the case where $B$ is elliptic of infinite order, we refer to \cite[Theorem 3]{babeca} which states that if $B$ does not commute with a matrix $C\in\SL$, then the semigroup $\langle f_B,f_C\rangle$ is dense in $\mathrm{PSL}(2,\R)$. Thus, since $B$ and $A_0$ do not commute, the semigroup $\langle B, A_0\rangle$ contains hyperbolic matrices that induce projective maps with distinct fixed points, and we are led to the same contradiction.\\
For the proof of \emph{(2)} assume that $\A^*$ does not contain any hyperbolic matrices. Since $\#K_\A>1$ by assumption, $\A^*$ has to contain at least two parabolic matrices, say $A_1$ and $A_2$, such that the projective maps $\phi_{A_1}$ and $\phi_{A_2}$ have distinct fixed points in $\R\P^1$. It is then easy to check that there exist $n,m\in\N$ such that $A_1^nA_2^m$ is a hyperbolic matrix, contradicting our assumption.
\end{proof}

\begin{thm}\label{attlimset}
Suppose $\A\subseteq \SL$ is such that $K_\A$ is non-empty. Then $\psi(K_\A)=\att(F_\A)$.
\end{thm}

\begin{proof}
Due to our assumption, $\A^*$ has to contain at least one hyperbolic or parabolic matrix, and so $\att(F_\A)$ is also non-empty. Observe that if the map $\phi_A$ fixes the point $x\in(0,\pi]$, for some $A\in\SL$, then equation \eqref{psieq} implies that $f_A$ fixes $\psi(x)\in\overline{\R}$. Hence, if $x\in K_\A$ then $\psi(x)$ is either the attracting fixed point of $f_A$, if the transformation is hyperbolic, or the unique fixed point of $f_A$, if the transformation is parabolic. It follows that $\psi(x)\in\att(F_\A)$, and so $\psi(K_\A)\subseteq \att(F_\A)$.\\
Note that if $\A^*$ contains a hyperbolic matrix then the other inclusion follows immediately from Corollary~\ref{att}. So, if $\#K_\A>1$ then we have nothing to prove due to Lemma~\ref{reduc}. If, on the other hand, $K_\A=\{x_0\}$, for some $x_0\in\R\P^1$ then all the hyperbolic and parabolic maps in $\Phi_\A$ have $x_0$ as their attracting or unique fixed point, in which case $\psi(K_\A)=\att(F_\A)=\{\psi(x_0)\}$.
\end{proof}

One can show that $K_\A$ is empty if and only if $\att(F_\A)$ is empty, and thus the equality $\psi(K_\A)=\att(F_\A)$ in Theorem \ref{attlimset} holds for any set $\A\subseteq \SL$. We omit the proof of this fact since we do not require this stronger version of the theorem.

Theorem~\ref{attlimset} allows us to utilise the theory of M\"obius semigroups in order to show that the attractor of a projective IFS $\Phi_\A$ shares many standard properties of attractors of iterated function systems.

\begin{thm}\label{dense}
Let $\A\subseteq \SL$ be such that $K_\A$ is non-empty. The attractor $K_\A$ of $\Phi_\A$ satisfies the following properties:
\begin{enumerate}
\item $K_\A=\Phi_\A(K_\A)$,
\item If $x\in K_\A$ is not fixed by every element of $\Phi_\A$, then $K_\A=\overline{\bigcup_{n\in\N} \Phi_{\A^n}(\{x\})}$,
\item If $K_\A$ is infinite, then it is a perfect set.
\end{enumerate}
\end{thm}

In particular, the first two parts of Theorem~\ref{dense} imply that when $\A\subseteq \SL$ is finite and uniformly hyperbolic, then the set $K_\A$ coincides with the attractor of $\Phi_\A$ as this was defined in \cite{bv} and \cite{solomyak} (see part \textit{(4)} in Theorem \ref{abybv}). Hence, our definition of an attractor generalises the definition used in \cite{bv}, to a wider class of, potentially infinite, systems.

If $C \subseteq \R\P^1$ has the property that $C=\Phi_\A(C)$ we say that it is invariant under $\Phi_\A$, or that it is an invariant set. It is interesting to note that one could easily find examples where $\Phi_\A$ admits more than one non-empty, closed, proper, invariant set $C$. However, if $\A$ is irreducible and $K_\A$ is non-empty, then $K_\A$ is the ``smallest" non-empty, closed, invariant set. Here, by smallest we mean that if $C \subseteq \R\P^1$ is non-empty, closed and invariant, then $K_\A\subseteq C$. This follows from the fact that since $\A$ is irreducible, any non-empty invariant set has to contain all attracting fixed points of hyperbolic maps and all unique fixed points of parabolic maps in $\langle \Phi_\A \rangle$. The statement obviously fails if $\A$ is reducible (take, for example, a set $\A$ that contains only one hyperbolic matrix). This property is reminiscent of the definition of an attractor of an IFS consisting of contractions of $\R^n$  as the unique, non-empty, compact, invariant set \cite{hutch}.

Before we proceed with the proof of Theorem~\ref{dense}, we require the following lemma, where a right-composition sequence in $F_\A$ is defined as a sequence $(f_{B_n})\subseteq \langle F_\A \rangle$, with $B_n=A_{i_1} A_{i_2} \cdots A_{i_n}$ for some $(A_{i_n})\subseteq \A$.

\begin{lma}\label{comp}
Let $\A\subseteq \SL$ be such that $\att(F_\A)\neq \emptyset$. Then $\att(F_\A)$ is the set of accumulation points in $\overline{\R}$ of the set
\[
\{f_{B_n}(i)\colon (B_n)\, \text{is a right-composition sequence in}\, F_\A\}.
\]
\end{lma}

\begin{proof}
Let $C$ be the set of accumulation points of
\[
\{f_{B_n}(i)\colon (B_n)\, \text{is a right-composition sequence in}\, F_\A\},
\]
in $\overline{\R}$. By definition, we have that $C\subseteq \att(F_\A)$. We now show the other inclusion.\\
If $\att(F_\A)$ is a singleton $\{x\}$, then $x$ is either the attracting fixed point of a hyperbolic transformation or the unique fixed point of a parabolic transformation, and the result follows from the Denjoy--Wolff theorem and the fact that the iterates of a M\"obius transformation form a right-composition sequence.\\
Suppose that $\#\att(F_\A)>1$. Then Theorem~\ref{attlimset} and Lemma~\ref{reduc} imply that the semigroup $\langle F_\A \rangle$ has to contain hyperbolic transformations. So, due to Corollary~\ref{att}, $\att(F_\A)$ is the closure of attracting fixed points of hyperbolic transformations in $\langle F_\A \rangle$. Note that if $x\in\overline{\R}$ is the attracting fixed point of a hyperbolic transformation $f_A\in\langle F_\A \rangle$, then $f_{A^n}(i)\in C$, for all $n$ and the sequence $(f_{A^n}(i))$ converges to $x$. Thus $C$ contains all attracting fixed points of hyperbolic transformations in $\langle F_\A \rangle$, and the result follows from the fact that $C$ is closed as the intersection of closed sets in $\widehat{\C}=\C\cup\{\infty\}$.
\end{proof}

\begin{proof}[Proof of Theorem~\ref{dense}]
Part \emph{(3)} follows immediately by combining Corollary~\ref{att} with Theorem~\ref{attlimset} and Lemma~\ref{reduc}. Due to Theorem~\ref{attlimset} and the fact that the function $\psi$ in \eqref{psieq} is a bijection, it suffices to prove the following two properties for $\att(F_\A)$:
\begin{enumerate}[label=\emph{(\arabic*')}]
\item $\att(F_\A)=\bigcup_{A\in\A}f_A(\att(F_\A))$,
\item If $x\in \att(F_\A)$ is not fixed by every element of $F_\A$, then $\att(F_\A)=\overline{ F_\A (x) }$.
\end{enumerate}
We start with \textit{(1')}. The inclusion $\bigcup_{A\in\A} f_A(\att(F_\A))\subseteq \att(F_\A)$ follows from the  forward invariance of $\att(F_\A)$. Suppose that $y\in \att(\A)$. Then, there exists a sequence $(B_n)\subseteq \A^*$, such that $f_{B_n}(i)$ converges to $y$, as $n\to\infty$. Due to Lemma~\ref{comp}, we can find $A_1 \in \A$ and  $(C_n)\subseteq \langle F_\A \rangle$ such that $(B_n)$ can be written as $B_n=A_1C_n$, for all $n\in\N$. Hence, $f_{B_n}(i)=f_{A_1}(f_{C_n}(i))$, for all $n\in\N$, and thus $f_{C_n}(i)$ converges to $f_{A_1}^{-1}(y)$, as $n\to\infty$. So the point $f_{A_1}^{-1}(y)$ lies in $\att(F_\A)$, implying that $y\in f_{A_1}(\att(\A))$ as required.\\
For the second property, note that for any $x\in\att(F_\A)$, we have that $\overline{ F_\A (x)}$ is contained in $\att(F_\A)$, by forward invariance. Suppose that $x$ is a point in $\att(F_\A)$ that is not fixed by every element of $F_\A$ and $y\in\att(F_\A)$. The proof is complete upon showing that $y\in \overline{ F_\A (x)}$. Take an open neighbourhood $U$ of $y$ in $\overline{\R}$. Consider the set $\A^{-1}\vcentcolon = \{A_1^{-1}, A_2^{-1},\dots, A_N^{-1}\}$ and observe that $\att(F_\A)=\rep(F_{\A^{-1}})$. So, if $S$ is the semigroup generated by $F_{\A^{-1}}$, \cite[Proposition~2.16]{frmast} states that the image $S(U)$ of $U$ under $S$ is either the whole of $\overline{\R}$, or $\overline{\R}$ missing one point. In the latter case, it is easy to check that the singleton $\overline{\R}\setminus S(U)$ is fixed by every transformation in $F_{\A^{-1}}$. Since $x$ is not fixed by every element of $F_{\A^{-1}}$, we have that, in either case, $x\in S(U)$. Hence, there exists $f_C\in S$ such that $x\in f_C(U)$, or equivalently $f_C^{-1}(x)\in U$. As $f_C^{-1}\in\langle F_\A \rangle$, and the neighbourhood $U$ of $y$ was chosen arbitrarily, we see that there exists a sequence $(f_{D_n})\subseteq\langle F_\A \rangle$ such that $f_{D_n}(x)$ converges to $y$, as $n\to\infty$.
\end{proof}

\subsection{The repeller}

Motivated by Theorem~\ref{attlimset} and Definition~\ref{deflim}, we make the following definition.

\begin{defn}
For $\A\subseteq \SL$, we define $R_\A$ to be the smallest closed set containing all repelling fixed points of hyperbolic maps in $\langle \Phi_\A \rangle$ and all unique fixed points of parabolic maps in $\langle \Phi_\A \rangle$. The set $R_\A$, if non-empty, is called the \emph{repeller} of $\Phi_\A$.
\end{defn}

Observe that the repeller of an IFS induced by a set $\A\subseteq \SL$ is the attractor of the IFS induced by $\A^{-1}\vcentcolon=\{A^{-1}\colon A\in\A\}$. So it is easy to check that many of our arguments in the proof of Theorem~\ref{attlimset} and the third part of Theorem~\ref{dense} carry over to the repeller of $\Phi_\A$ and yield the following lemma.

\begin{lma}\label{replimset}
Let $\A\subseteq \SL$. The repeller of $\Phi_\A$ satisfies the following properties.
\begin{enumerate}
\item $R_\A=\emptyset \iff K_\A=\emptyset$,
\item If $R_\A\neq \emptyset$, then $R_\A\vcentcolon =\psi^{-1}(\rep(F_\A))$.
\item If $R_\A$ is infinite, then it is a perfect set.
\end{enumerate}
\end{lma}

Notice that the attractor and the repeller of an IFS are not always disjoint. Evaluating their intersection provides us with a necessary and sufficient condition for $\A$ to be uniformly hyperbolic. This condition, although not explicitly stated, can be easily inferred from the material in \cite{aby} (see \cite[Subsection 2.4.2]{aby}).
\begin{lma}\label{uh2}
A finite set $\A\subseteq \SL$ is uniformly hyperbolic if and only if $\A$ is semidiscrete and $K_\A\cap R_\A=\emptyset$. 
\end{lma}
Lemma \ref{uh2} implies that if $\A$ is semidiscrete but not uniformly hyperbolic, then the intersection $K_\A\cap R_\A$ has to be non-empty. The simplest example of this situation is the existence of a parabolic matrix $A$ in $\A^*$, since then the unique fixed point of $\phi_A$ lies both in $K_\A$ and $R_\A$ (see Section \ref{eg} for further examples).

\subsection{The Furstenberg measure} \label{f-section}

Let $\A$ be a finite subset of $\SL$. For each $\i\in\I^\N$ and any $n\in\N$, we write $A_{\i\lvert n}=A_{i_1}A_{i_2} \cdots A_{i_n}$. Notice that for any $\i\in\I^\N$, the sequence of transformations $(f_{A_{\i\lvert n}})$ is a  right-composition sequence in $F_\A$. \cite[Theorem 1.3]{jasho} states that the sequence $(f_{A_{\i\lvert n}}(i))\subseteq \H$ converges for every $\i\in\I^\N$ if and only if $\A$ is semidiscrete. In addition, the limit of $(f_{A_{\i\lvert n}}(i))$ is a point in $\overline{\R}$. Hence the maps $\Pi \colon \I^\N \to \overline{\R}$ and $\Pi_\psi: \I^\N \to \R\P^1$ given by $\Pi(\i)=\lim_{n\to\infty}f_{A_{\i\lvert n}}(i)$ and $\Pi_\psi(\i)=\psi^{-1} \circ \Pi(\i)$ are well-defined if and only if the set $\A$ is semidiscrete. Note that for any $i \in \I$ and $\i \in \I^\N$, $f_{A_i}\circ \Pi(\i)=\Pi(\sigma_i(\i))$ where $\sigma_i:\I^\N \to \I^\N$ is the map $\sigma_i(\i)=i\i$. Furthermore, Lemma~\ref{comp} implies that $\bigcup_{\i\in\I^\N}\Pi(\i)=\att(F_\A)$, and so $\bigcup_{\i\in\I^\N}\Pi_\psi(\i)=K_\A$ by Theorem~\ref{attlimset}.

Suppose $\A$ is semidiscrete and irreducible. In the following short lemma we will show that this implies that $\A$ is strongly irreducible.

\begin{lma} \label{strong}
Suppose $\A \subseteq \SL$ is finite and semidiscrete. Then $\A$ is irreducible if and only if it is strongly irreducible.
\end{lma}

\begin{proof} It is sufficient to show that if $\A$ is semidiscrete and irreducible, then it is strongly irreducible. Suppose for a contradiction that $\A$ is irreducible but not strongly irreducible. Let $X \subseteq \R\P^1$ be a finite set which is preserved by the set of maps $\Phi_\A$. By the irreducibility of $\A$ there exists $A \in \A$ and $x \in X$ such that $x$ is not a fixed point of $\phi_A$. Since $\A$ is semidiscrete, $A$ cannot be an elliptic matrix and therefore the set $\{\phi_{A^n}(x)\}_{n \in \N} \subseteq \R\P^1$ is infinite. Moreover, $\phi_{A^n}x \in X$ for all $n \in \N$, by strong irreducibility. This is a contradiction since $X$ is finite.
\end{proof}

Next we construct the Furstenberg measure using the map $\Pi_\psi$. Given $\i \in \I^n$ let $[\i]$ denote the cylinder set $[\i]\vcentcolon=\{\j \in \I^\N: \j|_n=\i\}$. Cylinder sets are clopen and generate the topology on $\I^\N$. Let $(p_i)_{i \in \I}$ be some probability vector, where each $p_i>0$ and define a measure $\mu$ on cylinder sets by $\mu([i_1 \ldots i_n])=p_{i_1} \cdots p_{i_n}$. By Kolmogorov's extension theorem this defines a measure $\mu$ on $\I^\N$. Let $\nu$ denote the pushforward measure $\nu=(\Pi_\psi)^*(\mu)$, where $\nu(E)\vcentcolon= \mu(\Pi_\psi^{-1}E)$.  The measure  $\nu$ is a stationary measure, since for any measurable $E \subseteq \R\P^1$,
\begin{align*}
\sum_{i \in \I} p_i \nu(\phi_{A_i}^{-1} E) &= \sum_{i \in \I}p_i \mu((\phi_{A_i} \circ \psi^{-1}\circ \Pi)^{-1}E) \\
&= \sum_{i \in \I} p_i \mu((\psi^{-1} \circ f_{A_i} \circ \Pi)^{-1}E) \\
&= \sum_{i \in \I}p_i \mu((\psi^{-1} \circ \Pi \circ \sigma_i)^{-1}E) \\
&= \sum_{i \in \I} p_i \mu(\sigma_i^{-1}\circ\Pi_\psi^{-1} E)\\
&=\sum_{i \in \I} \mu(\Pi_\psi^{-1} E \cap [i]) =\mu(\Pi_\psi^{-1}E)=\nu(E).
\end{align*}

If $\A$ is semidiscrete, it generates an unbounded semigroup. If, in addition, $\A$ is irreducible, it must necessarily be strongly irreducible by Lemma \ref{strong} and therefore there is a unique Furstenberg measure. Hence if $\A$ is semidiscrete and irreducible, $\nu$ must be the unique Furstenberg measure (see \cite{fu} and \cite[Theorem 3.1]{hochman_sol}). Clearly $\nu$ is supported on $\bigcup_{\i\in\I^\N}\Pi_\psi(\i)=K_\A$, and therefore Corollary \ref{furst} follows immediately from Theorem \ref{MAIN}.

\section{Hausdorff dimension of reducible systems} \label{eg}
Our goal in this section is to establish the dimension formula from Theorem~\ref{MAIN} when $\A$ is reducible, as stated in the following proposition.

\begin{prop} \label{reducible}
Let $\A$  be a finite, Diophantine, semidiscrete and reducible subset of $\SL$. If $K_\A$ is not a singleton, then
\[
\hd K_\A=\min\{1,\delta_\A\}.
\]
\end{prop}

Suppose that $\A$ is as in the statement of Proposition~\ref{reducible} and let $x \in \R\P^1$ be the common fixed point of the maps in $\Phi_\A$.  If $x$ is a repelling fixed point of each map in $\Phi_\A$ then $\A$ is uniformly hyperbolic and Proposition~\ref{reducible} is simply a special case of Theorem~\ref{sol_thm}. On the other hand $x$ cannot be an attracting or parabolic fixed point for each map in $\Phi_\A$, since that would imply that $K_\A$ is a singleton. Therefore, $\A$ must contain one of the following cases:
\begin{enumerate}
\item either there exists a hyperbolic matrix $A$ and a parabolic matrix $B$ in $\A$, such that $\phi_B$ fixes the repelling fixed point of $\phi_A$; 
\item or there exist two hyperbolic matrices $A$ and $B$ in $\A$, such that the attracting fixed point of $\phi_A$ is the repelling fixed point of $\phi_B$.
\end{enumerate}

Thus the desired result of this section follows by combining our analysis of cases (1) and (2), which will be carried out in Lemmas~\ref{para-rep} and \ref{att-rep}, respectively. We note that the IFS studied in these lemmas are, in essence, elementary examples of projective IFS that do not exhibit any form of uniform contraction.

\begin{lma}\label{para-rep}
Let $\A$ be a finite subset of $\SL$. Assume that there exists a hyperbolic matrix $A\in\A^*$ and a parabolic matrix $B\in\A^*$, such that the repelling fixed point of $\phi_A$ is the unique fixed point of $\phi_B$. Then 
\[
\hd(K_\A)=\min\{1,\delta_\A\}=1.
\]
\end{lma}

\begin{proof}
We are first going to show that the attractor of $\{\phi_A,\phi_B\}$ contains an interval, and so $\hd(K_\A)=1$ because $\{\phi_A,\phi_B\}\subseteq\Phi_\A$.\\
Let $x_0$ be the common fixed point of $\phi_A$ and $\phi_B$. By conjugating $\A$ with an $\SL$ matrix we can assume that $A$ and $B$ have the form
\begin{align*}
&A=\begin{pmatrix}\lambda & 0\\
0 & \frac{1}{\lambda}
\end{pmatrix}
&\text{and}&
&B=\begin{pmatrix}
1 & 1\\
0 & 1
\end{pmatrix},
\end{align*}
for some $\lambda\in(0,1)$. Then, for all $n,m\in\N$,
\begin{equation}\label{b}
A^nB^m=\begin{pmatrix}
\lambda^n & \lambda^nm \\
0 & \frac{1}{\lambda^n}
\end{pmatrix},
\end{equation}
or, considering the action of $\A$ in the upper half-plane, $f_{A^nB^m}(z)=\lambda^{2n}z+\lambda^{2n}m$. Note that $f_{A^nB^m}$ is a hyperbolic transformation whose attracting and repelling fixed points are $\lambda^{2n}m/(1-\lambda^{2n})$ and infinity, respectively. Hence, for every $x\in[0,\infty)$ and all $\epsilon>0$, there exist positive integers $n,m$, such that 
\[
\lvert f_{A^nB^m}(i)-x \rvert<\epsilon.
\]
Thus $[0,\infty)\subseteq \att(\{f_A,f_B\})$, and Theorem~\ref{attlimset} yields the desired conclusion.\\
In order to complete the proof, we now show that $\min\{1,\delta_\A\}= 1$. For any $n\in\N$ define the integer $m(n)\vcentcolon= \lceil \tfrac{1}{\lambda^{2n}}\rceil$, and observe that for all $m<m(n)$, we have $m\lambda^n\leq \tfrac{1}{\lambda^n}$. Hence, \eqref{b} implies that $\norm{A^nB^m}=\lambda^{-n}$, for all $m<m(n)$, where $\norm{\cdot}$ was chosen to be the maximum norm in $\SL$. Now, for each $n\in\N$, define
\[
D_n=\{A^n,A^nB,A^nB^2,\dots , A^nB^{m(n)-1}\}.
\]
The sets $D_n$ are pairwise disjoint and each $D_n$ consists of $m(n)$ matrices from $\A^*$ whose norm is equal to $\lambda^{-n}$. Hence, for any $s>0$ we have that
\begin{equation}\label{z}
\sum_{n=1}^\infty\sum_{C\in D_n}\norm{C}^{-2s}\leq \zeta_\A(s).
\end{equation}
But 
\[
\sum_{n=1}^\infty\sum_{C\in D_n}\norm{C}^{-2s}=\sum_{n=1}^\infty m(n)\lambda^{2ns}\geq \sum_{n=1}^\infty \lambda^{(s-1)2n},
\]
and the sum on the right-hand side of this last inequality diverges for all $s<1$. Hence, inequality \eqref{z} implies that $\delta_\A\geq 1$ and so $\min\{1,\delta_\A\}=1$.
\end{proof}

\begin{lma}\label{att-rep}
Let $\A$ be a finite and semidiscrete subset of $\SL$. Assume that there exist hyperbolic matrices $A,B\in\A^*$, such that the attracting fixed point of $\phi_A$ is the repelling fixed point of $\phi_B$. Then
\[
\hd(K_\A)=\min\{1,\delta_\A\}=1.
\]
\end{lma}

\begin{proof}
By assumption we have that $a(A)=r(B)$, where the $a(A)$ is the attracting fixed point of $\phi_A$ and $r(B)$ is the repelling fixed point of $\phi_B$. If $r(A)= a(B)$, where $r(A)$ is the repelling fixed point of $\phi_A$ and $a(B)$ is the attracting fixed point of $\phi_B$, then $A$ and $B$ can be simultaneously conjugated to diagonal matrices 
\begin{align*}
&A=\begin{pmatrix}\lambda & 0\\ 0 & \frac{1}{\lambda}\end{pmatrix} &\text{and}& & B=\begin{pmatrix}\kappa & 0\\ 0 & \frac{1}{\kappa}\end{pmatrix},
\end{align*}
where $\lambda<1<\kappa$, and it is easy to check that $\mathrm{Id}\in \overline{\A^*}$, which is a contradiction. Hence, we have that $r(A)\neq a(B)$ and we can use \cite[Lemma~10.4]{jasho} in order to deduce that $\att(\{f_A,f_B\})$ is an interval and $\{A,B\}^*$ is not discrete (recall the definition of discreteness from Section~\ref{examples}). Theorem \ref{attlimset} tells us that $K_\A$ contains an interval and so $\hd(K_\A)=1$. Finally, in the second part of Section~\ref{examples} we showed that if $\A^*$ is not discrete then $\delta_\A=\infty$, which concludes our proof.
\end{proof}

\section{Hausdorff dimension of irreducible systems}\label{irred}

Throughout this section we assume that $\A$ is a finite, irreducible and semidiscrete subset of $\SL$. Note that the irreducibility of $\A$ implies that $K_\A$ is not a singleton. We will begin by constructing a uniformly hyperbolic subset of $\A^*$ whose attractor, in some sense, captures the dimension of $K_\A$. To this end we extend the Definition \ref{uh} of uniform hyperbolicity to countable sets of matrices in the obvious way:

We say that a countable set of matrices $\mathcal{B}$ is \emph{uniformly hyperbolic} if there exist real numbers $\lambda>1$ and $c>0$, such that for every $n\in\N$,
\[
\norm{A}\geq c\lambda^n, \quad \text{for all} \quad A\in\mathcal{B}^n.
\]

\subsection{Uniformly hyperbolic subsystems}

Let $\A$ be a finite, irreducible and semidiscrete subset of $\SL$.  In this subsection we construct a (countable) uniformly hyperbolic subset of $\A^*$. This construction will play an integral role in our proofs, since it will allow us to approximate the dimension of $K_\A$ by the dimension of attractors of (finite) uniformly hyperbolic subsystems, whose dimensions we can compute by Theorem \ref{sol_thm}. 

We begin by recording the following useful lemma, which is an immediate corollary of \cite[Theorem~13.1]{jasho}.

\begin{lma}\label{hyp}
Let $\A$ be an irreducible and semidiscrete subset of $\SL$. For every pair of open subsets $U$ and $V$ of $\R\P^1$ such that $U$ meets $K_\A$ and $V$ meets $R_\A$, there is a hyperbolic matrix $A\in\A^*$, such that the map $\phi_A$ has attracting fixed point in $U$ and repelling fixed point in $V$.
\end{lma}

\begin{proof}
Note that if $\A$ is semidiscrete and irreducible, it is necessarily strongly irreducible by Lemma \ref{strong}. So, because $\A^*$ does not contain any elliptic matrices, the semigroup of M\"obius tranformations $\langle F_\A \rangle$ does not have a finite orbit in $\overline{\H}$. \cite[Theorem~13.1]{jasho} states that if $\langle F_\A \rangle$ does not have a finite orbit in $\overline{\H}$, then for every pair of open subsets $U$ and $V$ of $\overline{\R}$ such that $U$ meets $\att(F_\A)$ and $V$ meets $\rep(F_\A)$, there is a hyperbolic matrix $A\in\A^*$, such that the transformation $f_A$ has attracting fixed point in $U$ and repelling fixed point in $V$. The result follows immediately from Theorem~\ref{attlimset}
\end{proof}

The key obstruction to extending Theorem \ref{sol_thm} to the semidiscrete setting is the fact that $K_\A \cap R_\A \neq \emptyset$, unlike the uniformly hyperbolic setting where this intersection is always empty (see Lemma \ref{uh2}). This makes the projective IFS $\Phi_\A$ more difficult to study, due to the non-uniformity of the contraction of the maps in $\langle \Phi_\A \rangle$ near the points of intersection. However, in the next proposition we show that under our standing assumption that $\A$ is irreducible and semidiscrete, we can find an open subset of $\R\P^1$ which only intersects the attractor, and another open subset of $\R\P^1$ which only intersects the repeller. Then, employing Lemma \ref{hyp}, we construct a uniformly hyperbolic subset of $\A^*$ which will play a key role in the proof of Theorem \ref{MAIN} in the irreducible case.

\begin{prop} \label{uh-sub}
Suppose that $\A$ is a finite, irreducible and semidscrete subset of $\SL$. Then there exists $A_0 \in \A^*$ and open intervals $U, U', V \subsetneq \R\P^1$, with $\overline{U'} \subsetneq U$ and $\overline{U}\cap\overline{V}=\emptyset$, such that for all $B \in \A^*$:
\begin{enumerate}
\item  $\phi_B\left(\overline{U}\right) \subsetneq \R\P^1 \setminus V$, and
\item$\phi_{A_0}\left(\R\P^1 \setminus V\right) \subsetneq U'$.
\end{enumerate}
In particular, the countable set
\[
\Gamma\vcentcolon=\{A_0B\colon B \in \A^*\},
\]
is uniformly hyperbolic.
\end{prop}

\begin{proof}Since $\A$ is irreducible and semidiscrete, it is necessarily strongly irreducible by Lemma \ref{strong}, therefore the attractor and the repeller of $\Phi_\A$ are perfect sets and thus uncountable (see the third parts of Theorem~\ref{dense} and Lemma~\ref{replimset}). We first prove that there exist open intervals $U,V\subsetneq \R\P^1$, such that $U\cap K_\A \neq \emptyset$ while $U\cap R_\A=\emptyset$, and similarly, $V\cap R_\A\neq \emptyset$ while $V\cap K_\A=\emptyset$. The proof will be carried out for $V$; the case for $U$ is similar. Note that due to Theorem~\ref{attlimset} and the definition of $R_\A$, it suffices to prove that there exists an open interval $V\subsetneq\overline{\R}$, such that $V\cap \rep(F_\A)\neq \emptyset$ while $V\cap \att(F_\A)=\emptyset$. Suppose that there exists no such $V$. Then, if $x\in\rep(F_\A)$, any open neighbourhood $D$ of $x$ in $\overline{\R}$ has to intersect $\att(F_\A)$. Hence, for every $x\in\rep(F_\A)$ there exists a sequence $(y_n)\subseteq \att(F_\A)$ that converges to $x$. Because $\att(F_\A)$ is closed, we obtain that $\rep(F_\A)\subseteq \att(F_\A)$. We now refer to \cite[Theorem~15.2]{jasho}, which states that if $\mathrm{Id}\notin \partial\A^*$ and $\rep(F_\A)$ is not a singleton, then $\rep(F_\A)\subseteq \att(F_\A)$ implies that $\A^*$ is a group. This yields a contradiction in our case, since $\mathrm{Id}\notin \A^*$. Thus we obtain the desired open sets $V$ and $U$.\\
In addition, as the attractor and the repeller are perfect sets, we can choose $U=(x_1,x_2)$, where $x_1,x_2$ lie in $K_\A$ and not in $R_\A$, and $V=(y_1,y_2)$, where $y_1,y_2$ lie in $R_\A$ and not in $K_\A$. That is, $U$ and $V$ can be chosen such that $U\cap K_\A \neq \emptyset$ while $\overline{U}\cap R_\A=\emptyset$, and similarly, $V\cap R_\A\neq \emptyset$ while $\overline{V}\cap K_\A=\emptyset$.
Lemma~\ref{hyp} now implies that we can find a hyperbolic matrix $A_0\in\A^*$, such that the attracting fixed point of $\phi_{A_0}$ lies in $U$ and its repelling fixed point lies in $V$. By considering iterates of $A_0$ we can also assume that $\phi_{A_0}(\R\P^1\setminus V)\subsetneq U$. Also, let $n_0\in\N$ be such that $A_0\in\A^{n_0}$. Take $B\in\A^*$ and suppose that $\phi_B(U)\cap V\neq\emptyset$. Since each of the intersections $\overline{U}\cap K_\A$ and $\overline{V}\cap R_\A$ contains at least three points (two points on the boundary and at least one in the interior), there exist points $x,y$ such that $\phi_B(x)=y$, where either $x\in U$ and $y\in R_\A$, or $x\in K_\A$ and $y\in V$. Both possibilities yield a contradiction due to our choice of $U$ and $V$ and the fact that $K_\A$ is forward invariant, while $R_\A$ is backward invariant. We conclude that $\phi_B(U) \subsetneq \R\P^1\setminus V$ for all $B\in\A^*$. Let $U'$ be any open interval which contains $\phi_{A_0}(\R\P^1 \setminus V)$ such that $\overline{U'} \subsetneq U$. Then, for every $B\in\A^*$,  $\phi_{A_0B}(\overline{U}) \subsetneq U'$.\\
It remains to prove that $\Gamma$ is uniformly hyperbolic, which we show by following the arguments from \cite[Section~2.2]{aby}. Let $d_U$ and $d_{U'}$ denote the Hilbert metrics in $U$ and $U'$ respectively. Since $U'$ is compactly contained in $U$, there exists $\lambda>1$, depending only on $U$ and $U'$, such that for all $x,y\in U'$, we have that $d_U (x,y)\leq \lambda^{-1} d_{U'}(x,y)$. This implies that for every $M\in \Gamma$ and all $x,y\in U$
\[
d_U (\phi_M(x),\phi_M(y))\leq \lambda^{-1} d_{U'}(\phi_M(x),\phi_M(y))\leq d_U(x,y),
\]
where in the last inequality we used the fact that the map $\phi_M\colon (U,d_u)\to (U',d_{u'})$ is a contraction. Thus, if $B_n=M_1M_2\dots M_n\in \Gamma^n$, for $n=1,2,\dots, N$, then for all $x,y\in U$
\begin{align}
 d_U(\phi_{B_N}(x),\phi_{B_N}(y))&\leq \lambda^{-1} d_{U}(\phi_{B_{N-1}}(x),\phi_{B_{N-1}}(y)) \nonumber \\
&\leq \lambda^{-1} d_{U'}(\phi_{B_{N-1}}(x),\phi_{B_{N-1}}(y))\leq \dots \leq \lambda^{-N}d_U(x,y).\label{d}
\end{align}
The metric $d_U$ restricted to $U'$ is comparable to the Euclidean metric $d_e$ of $\R\P^1$. So, inequality \eqref{d} implies that
\[
d_e(\phi_{B_N}(x),\phi_{B_N}(y))\leq C\lambda^{-N} d_e(x,y),
\]
for some constant $C>0$. Hence, $\norm{B_N}\geq C^{-\frac{1}{2}}\lambda^{\frac{N}{2}}$, as required.
\end{proof}

Given a finite or countable subset $\A \subseteq \SL$ we can define the \emph{pressure function} $P_\A:[0,\infty) \to \overline{\R}$ as
\begin{equation}
P_\A(s)=\lim_{n \to \infty} \left(\sum_{A \in \A^n} \norm{A}^{-2s}\right)^{\frac{1}{n}}. \label{pressure}
\end{equation}

Due to the submultiplicativity of the norm, the limit in \eqref{pressure} exists (although it may equal infinity) and equals the supremum of the expressions on the right hand side of \eqref{pressure}. It is easy to see that $P_\A(s)$ is a decreasing function of $s$, and that it is convex and continuous on $( s_*,\infty)$, where $s_*\vcentcolon= \inf\{s \geq 0: P_\A(s)<\infty\}$. We define the \emph{minimal root} $s_\A$ of $P_\A$ as $s_\A=\infty$ if $P_\A(s)>1$ for all $s>0$ and
\begin{equation*}\label{root}
s_\A= \inf\{s \geq 0: P_\A(s) \leq 1\}
\end{equation*}
otherwise. Throughout, we shall refer to the minimal root of the pressure function as simply the \emph{root}.

Pressure functions and their roots are more useful than zeta functions and their critical exponents for obtaining lower bounds on the Hausdorff dimension; see for instance the proof of the lower bound in \cite{solomyak}. Switching between $s_\A$ and $\delta_\A$ poses no problem in the uniformly hyperbolic setting, since in this case $P_\A$ must be \emph{strictly} decreasing, so by the root test one can fairly easily show that in fact $s_\A=\delta_\A$. De Leo \cite[Theorem 2]{deleo1} extended this by showing that $s_\A=\delta_\A$ when $\A$ belongs to a special class of systems containing parabolic matrices. His conditions however are somewhat restrictive since they require, for example, $\A$ to generate a discrete and free semigroup \footnote{One can show that the main condition throughout \cite{deleo1}, that of a system being ``fast parabolic", corresponds to a finite set of hyperbolic and parabolic matrices $\A \subseteq \SL$ with an invariant multicone such that the intersection $K_\A \cap R_\A$ only contains fixed points of parabolic maps from $\A$, and such that there do not exist distinct points $x, y \in K_\A \cap R_\A$ and $A \in \A^*$ for which $\phi_A(x)=y$.}. We extend this to any semidiscrete and irreducible subset of $\SL$.

\begin{thm} \label{s=d}
Let $\A$ be a finite, semidiscrete and irreducible subset of $\SL$. Then $s_\A=\delta_\A$.
\end{thm}

\begin{proof}
Observe that for any finite or countable set $\mathcal{B}\subseteq \SL$, we have that $s_\mathcal{B}\leq \delta_\mathcal{B}$. By Proposition \ref{uh-sub} there exists $A_0 \in \A^{k}$, for some $k\in\N$, such that $\{A_0B: B \in \A^*\}$ is uniformly hyperbolic. Since $\delta_\A= \delta_{\A^{k}}$ and $s_{\A}=s_{\A^{k}}$ for all $k$, without loss of generality we can assume that $k=1$.\\
Define 
\[
\A_\infty\vcentcolon=\{A_0B: B \in (\A \setminus \{A_0\})^*\} \cup \{A_0\},
\]
and
\[
\A_N\vcentcolon= \A_\infty \cap \left(\bigcup_{m=1}^N \A^m\right).
\]
We first claim that $\lim_{N\to\infty}P_{\A_N}(s)=P_{\A_\infty}(s)$. Note that the limit exists because $P_{\A_N}(s) \leq P_{\A_{N+1}}(s) \leq P_{\A_\infty}(s)$, for all $N\in\N$. In particular, $\lim_{N\to\infty}P_{\A_N}(s) \leq P_{\A_\infty}(s)$, so we only need to show that $P_{\A_\infty}(s) \leq \lim_{N\to\infty}P_{\A_N}(s)$. To that end it suffices to prove that for any $q>0$ with $q<P_{\A_\infty}(s)$, we have that $q<\lim_{N\to\infty}P_{\A_N}(s)$.\\
Since $P_{\A_\infty}(s)=\sup_{n \in \N} \left(\sum_{A\in \A_\infty^n} \norm{A}^{-2s}\right)^{\frac{1}{n}}$, there exists $n_0$ sufficiently large that
\[
\left(\sum_{A\in\A_\infty^{n_0}} \norm{A}^{-2s}\right)^{\frac{1}{n_0}}>q.
\]
Also, as
\[
\lim_{N \to \infty} \left( \sum_{A\in \A_N^{n_0}} \norm{A}^{-2s}\right)^{\frac{1}{n_0}}= \left(\sum_{A\in\A_\infty^{n_0}} \norm{A}^{-2s}\right)^{\frac{1}{n_0}},
\]
we can choose $N_0$ sufficiently large that
\[
\left( \sum_{A\in\A_{N_0}^{n_0}} \norm{A}^{-2s}\right)^{\frac{1}{n_0}}>q.
\]
Therefore
\[
q<\left( \sum_{A\in\A_{N_0}^{n_0}} \norm{A}^{-2s}\right)^{\frac{1}{n_0}} \leq \sup_{n \in \N} \left( \sum_{A\in\A_{N_0}^n} \norm{A}^{-2s}\right)^{\frac{1}{n}}=P_{\A_{N_0}}(s) \leq \lim_{N\to\infty}P_{\A_N}(s).
\]
Hence, $P_{\A_\infty}(s)=\lim_{N\to\infty}P_{\A_N}(s)=\sup_N P_{\A_N}(s)$, concluding the proof of our claim.\\
We will show that
\[
\mbox{\textbf{(i)}}   \ \ \delta_\A=\delta_{\A_\infty},  \ \ \ \ \ \  
\mbox{\textbf{(ii)}}  \ \ \delta_{\A_\infty}=s_{\A_\infty},  \ \ \ \ \ \ 
\mbox{\textbf{(iii)}} \ \ s_{\A_{\infty}}=\sup_N s_{\A_N}, \ \ \ \ \ \ 
\mbox{\textbf{(iv)}} \ \sup_N s_{\A_N} =s_\A. \ \ \ \ \ \ 
\]
Proof of  \textbf{(i)}. Since $\A_\infty^* \subseteq \A^*$, we have $\delta_{\A_\infty} \leq \delta_\A$. Also,
\begin{align*}
\zeta_\A(s) &= \sum_{n=1}^{\infty} \sum_{\substack{\i \in \I^n \\ i_1=0}} \norm{A_\i}^{-2s}+\sum_{n=1}^{\infty} \sum_{\substack{\i \in \I^n \\ i_1\neq 0}} \norm{A_\i}^{-2s}\\
&\leq \zeta_{\A_\infty}(s) + \norm{A_0}^{2s}   \sum_{n=1}^{\infty} \sum_{\substack{\i \in \I^n \\ i_1\neq 0}} \norm{A_0A_\i}^{-2s}\\
&= (1+\norm{A_0}^{2s}) \zeta_{\A_\infty}(s).
\end{align*}
Hence we also have $\delta_\A \leq \delta_{\A_\infty}$.\\
Proof of \textbf{(ii)}. Since $s_{\A_\infty}\leq \delta_{\A_\infty}$, we can assume that $s_{\A_\infty}<\infty$ because otherwise there is nothing to prove. From the root test we can see that $P_{\A_\infty}(s) \geq 1$, for $s \leq \delta_{\A_\infty}$, and $P_{\A_\infty}(s) \leq 1$, for $s \geq \delta_{\A_\infty}$. So $\delta_{\A_\infty}$ is a root of $P_{\A_\infty}$.  Recall the uniformly hyperbolic set
\[
\Gamma=\{A_0B\colon B\in\A^*\},
\]
from Proposition \ref{uh-sub}. Note that since $\A_\infty \subseteq \Gamma$, the countable set $\A_\infty$ is uniformly hyperbolic. Let $\lambda>1$ such that for any $A \in \A_\infty^n$, $\norm{A} \geq \lambda^n$, for all $n$ large enough. We will show that this uniform hyperbolicity implies that $P_{\A_\infty}$ has a unique root. To see this, let $s=s_{\A_\infty}+\epsilon>s_{\A_\infty}$. Then
\[
\sum_{A \in \A_\infty^n} \norm{A}^{-2(s_{\A_\infty}+\epsilon)} \leq \frac{1}{\lambda^{2n\epsilon}} \sum_{A \in \A_\infty^n} \norm{A}^{-2s_{\A_\infty}}.
\]
In particular, $P_{\A_\infty}(s) \leq \frac{1}{\lambda^{2\epsilon}}\, P_{\A_\infty}(s_{\A_\infty})=\frac{1}{\lambda^{2\epsilon}}<1$. Therefore $\delta_{\A_\infty}$ is the unique root of $P_{\A_\infty}$, that is, $s_{\A_\infty}=\delta_{\A_\infty}$ as required.\\
Proof of  \textbf{(iii)}. Since $\A_N \subseteq \A_\infty$, clearly $\sup_N s_{\A_N} \leq s_{\A_\infty}$. Now, let $s< s_{\A_\infty}$ so that $P_{\A_\infty}(s)= \sup_n \left(\sum_{A \in A_\infty^n}\norm{A}^{-2s}\right)^{\frac{1}{n}}>1$. We can choose $n \in \N$ sufficiently large so that $\left(\sum_{A \in \A_\infty^n} \norm{A}^{-2s}\right)^{\frac{1}{n}}>1$. Moreover, since $\lim_{N \to \infty} \left(\sum_{A \in \A_N^n} \norm{A}^{-2s}\right)^{\frac{1}{n}}=\left(\sum_{A \in \A_\infty^n} \norm{A}^{-2s}\right)^{\frac{1}{n}}$, we can choose $N \in \N$ sufficiently large that $\left(\sum_{A \in \A_N^n} \norm{A}^{-2s}\right)^{\frac{1}{n}}>1$, so in particular $P_{\A_N}(s)>1$. Therefore $s< \sup_N s_{\A_N}$, that is, $s_{\A_\infty} \leq \sup_N s_{\A_N}$.\\
Proof of \textbf{(iv)}. Note that since $s_\A \leq \delta_\A= \sup_N s_{\A_N}$ by the conditions \textbf{(i-iii)} above, it is sufficient to prove that $\sup_N s_{\A_N} \leq s_\A$. For a contradiction we assume that $s_{\A} < s_{\A_N}$ for some $N \in \N$. Then we can choose $C>1$ and $n_0 \in \N$ such that 
\begin{equation}\label{C}
\left(\sum_{A \in \A_N^n} \norm{A}^{-2s_\A}\right)^{\frac{1}{n}} \geq C,
\end{equation}
for all $n \geq n_0$. Let $\epsilon>0$ and $n_1 \geq n_0$ such that $C^{\frac{1}{N}}(1-\epsilon) \norm{A_1}^{-\frac{2s_\A}{n_1}}>1$. Note that there exists $M \in \N$ such that for all $m \geq M$, 
\begin{equation}\label{eps}
\left(\sum_{A \in \A^m} \norm{A}^{-2s_\A} \right)^{\frac{1}{m}} \geq 1-\epsilon. 
\end{equation}
Given $A \in \A_N^k$, write $|A|=l$ if $A\in\A^l$. Now, observe that for any $k \in \N$,
\begin{align}\label{rewrite}
\sum_{A \in \A^{Nk}} \norm{A}^{-2s_\A} \geq &\sum_{\substack{A \in \A_N^k\\ |A|=Nk}} \norm{A}^{-2s_\A}+ \sum_{\substack{A \in \A_N^k\\ |A|=Nk-1}} \norm{AA_0}^{-2s_\A} \,+ \nonumber\\
&+  \sum_{\substack{A \in \A_N^k\\ k \leq |A|\leq Nk-2}} \sum_{\i \in \I^{Nk-1-|A|}} \norm{AA_0A_\i}^{-2s_\A}. 
\end{align}
Put $c= \min_{1 \leq m \leq M-1} \sum_{\i \in \I^m} \norm{A_\i}^{-2s_\A},$ and observe that by \eqref{eps} we have that for any $l \leq Nk-2$,
\[
\sum_{\i \in \I^{Nk-1-l}} \norm{A_\i}^{-2s_\A} \geq \min \{c, (1-\epsilon)^{Nk-1-l}\} \geq \min\{c, (1-\epsilon)^{Nk}\}=(1-\epsilon)^{Nk},\]
whenever $k \geq k_0$ for some sufficiently large $k_0$. Finally, fix $k \geq \min\{n_1, k_0\}$. Then, from (\ref{rewrite}) we obtain,
\begin{align*}
\sum_{A \in \A^{Nk}} \norm{A}^{-2s_\A} &\geq \sum_{\substack{A \in \A_N^k\\ |A|=Nk}} \norm{A}^{-2s_\A}+ \norm{A_0}^{-2s_\A}\left(\sum_{\substack{A \in \A_N^k\\ |A|=Nk-1}} \norm{A}^{-2s_\A}\right) +\\
& +  \norm{A_0}^{-2s_\A}\left(\sum_{\substack{A \in \A_N^k\\ k \leq |A|\leq Nk-2}} \norm{A}^{-2s_\A}\left(\sum_{\i \in \I^{Nk-1-|A|}} \norm{A_\i}^{-2s_\A}\right)\right) \\
&\geq \norm{A_0}^{-2s_\A}(1-\epsilon)^{Nk}\sum_{A \in \A_N^k} \norm{A}^{-2s_\A}.
\end{align*}
Thus by (\ref{C}),
\begin{align*}
\left(\sum_{A \in \A^{Nk}} \norm{A}^{-2s_\A} \right)^{\frac{1}{Nk}} &\geq  \norm{A_0}^{-\frac{2s_\A}{Nk}}(1-\epsilon)\left(\sum_{A \in \A_N^k} \norm{A}^{-2s_\A}\right)^{\frac{1}{Nk}}\\
&\geq C^{\frac{1}{N}} (1-\epsilon) \norm{A_0}^{-\frac{2s_\A}{Nk}} >1.
\end{align*}
In particular, $P_{\A^{N}}(s_{\A}) >1$, which is a contradiction because $s_\A=s_{\A^N}$. Thus the proof of \textbf{(iv)} is complete.
\end{proof}

\subsection{A lower bound for the dimension}

In this subsection, using Proposition \ref{uh-sub} and Theorem \ref{s=d} we establish a lower bound for the dimension formula in Theorem \ref{MAIN}, for irreducible $\A$.

\begin{prop}\label{lowerbound}
Let $\A$ be a finite, irreducible, Diophantine  and semidiscrete subset of $\SL$. Then
\[
\min\{1,\delta_\A\}\leq \hd(K_\A).
\]
\end{prop}

\begin{proof}
 Let $A_0$ and $\Gamma$ be given by Proposition \ref{uh-sub}. For each $n\in\N$ define the following subset of $\Gamma$,
\[
\Gamma_n=\{A_0B\colon B \in \A^n\}.
\]
For each $n \in \N$, $K_{\Gamma_n}$ is not a singleton. This is because the attracting fixed points of matrices in $\Gamma_n$ cannot all coincide by irreducibility of $\A$ and the images of these fixed points under the map $\phi_A$ all belong to $K_{\Gamma_n}$. Note that the attractor of each $\Gamma_n$ is contained in $K_\A$ and so $\hd(\cup_nK_{\Gamma_n})\leq \hd(K_\A)$. In order to conclude the proof, it suffices to show that $\min\{1,\delta_\A\}=\hd(\cup_n K_{\Gamma_n})$. Observe that all matrices in $\Gamma_n$ have the same length as words in $\A$. As $\A$ is Diophantine, it is easy to check that $\Gamma_n$ is also Diophantine, for all $n$. Furthermore, each $\Gamma_n$ is uniformly hyperbolic as a subset of the uniformly hyperbolic set $\Gamma$. Thus Theorem~\ref{sol_thm} is applicable and yields that $\hd(K_{\Gamma_n})=\min\{1,\delta_{\Gamma_n}\}$. The countable stability of the Hausdorff dimension now implies that 
\[
\hd(\cup_n K_{\Gamma_n})=\sup_n\hd(K_{\Gamma_n})=\min\{1,\sup_n \delta_{\Gamma_n}\}.
\]
Thus, our goal is to show that $\sup_n\delta_{\Gamma_n}=\delta_\A$. Since $\Gamma_n^*\subseteq \A^*$, we have that $\delta_{\Gamma_n} \leq\delta_{\A}$ and therefore $\sup_n \delta_{\Gamma_n} \leq \delta_\A$. For the other inequality, recall that by Theorem \ref{s=d} it is sufficient to show that $\sup_n s_{\Gamma_n} \geq s_\A$. Observe that the submultiplicativity of the norm implies that for all $k\in\N$, and all $s>0$
\begin{equation} \label{lb-eq}
\norm{A_0}^{-2s}\sum_{B\in\A^k}\norm{B}^{-2s}\leq \sum_{B\in\A^k}\norm{A_0B}^{-2s}= \sum_{C\in\Gamma_k}\norm{C}^{-2s}.
\end{equation}
By definition of the pressure function, for any $s > \sup_n s_{\Gamma_n}$ and $k \in \N$,
\[
\sum_{C\in\Gamma_k}\norm{C}^{-2s} \leq P_{\Gamma_k}(s) \leq 1.
\]
Therefore by (\ref{lb-eq}) we have
\[
P_\A(s) =\lim_{k \to \infty} \left(\norm{A_0}^{-2s}\sum_{B\in\A^k}\norm{B}^{-2s}\right)^{\frac{1}{k}} \leq 1.
\]
So $s_\A \leq s$, implying that $s_\A \leq \sup_n s_{\Gamma_n}$, as required.
\end{proof}

\subsection{An upper bound for the dimension}

In this section we obtain the upper bound in Theorem \ref{MAIN} for irreducible $\A$, thus settling Theorem \ref{MAIN} in this case.

To obtain an upper bound on the Hausdorff dimension we need to estimate the Hausdorff measure of a natural cover, for which we require some control over the derivative of maps in $\langle \Phi_\A\rangle$. In principle, the derivative of each map $\phi_A$ is bounded between $\norm{A}^{-2}$ and $\norm{A}^2$, and the derivative attains both bounds at some point in $\R\P^1$ -- see \cite[\S 2.4]{hochman_sol}. However, outside of a small part of $\R\P^1$ the map $\phi_A$ exhibits strong contraction properties which is the content of the next lemma, taken from \cite[Lemma 2.4]{hochman_sol}. 

Given a matrix $A \in \SL$, the matrix $A^TA$ has two eigenvalues given by $\norm{A}^2$ and $ \norm{A}^{-2}$. We let $u_A^-$ be the eigenvector corresponding to the eigenvalue $\norm{A}^{-2}$. 

\begin{lma} \label{contract}
For any $\epsilon>0$, there exists $C_\epsilon>1$ such that for any $A \in \mathrm{SL}_2(\R)$
\[
\norm{A}^{-2}\leq|\phi_A'(x)| \leq C_\epsilon \norm{A}^{-2}\quad \text{for all} \quad x \in \R\P^1 \setminus (u_A^- - \epsilon,u_A^-+ \epsilon).
\]
\end{lma}

Lemma~\ref{contract} yields the following useful corollary (taken from \cite[Lemma 3.2]{solomyak}), where $\lvert F \rvert$ denotes the Lebesgue measure of a set $F\subseteq \R\P^1$.

\begin{lma} \label{pi lemma}
Let $U \subsetneq [0,\pi)$ be an open set with $\lvert U \rvert< \pi$. Then, for every $\epsilon>0$ there exists $C_\epsilon>1$ such that for any $A \in \SL$ with $(u_A^- -\epsilon, u_A^- +\epsilon) \subseteq U$, we have
\[
\pi-C_\epsilon\norm{A}^{-2}<\lvert \phi_A(U)\rvert<\pi.
\]
\end{lma}

For the rest of this section we assume that $\A$ is a finite, irreducible and semidiscrete subset of $\SL$. Proposition~\ref{uh-sub} is now applicable and yields a matrix $A_0 \in \A^*$ such that the set $\Gamma=\{A_0B\colon B\in\A^*\}$ is uniformly hyperbolic. In particular, there exist open sets $U,U'$, with $\overline{U'}\subsetneq U$ and $\overline{U}\cap R_\A=\emptyset$, such that $\overline{A(U)}\subsetneq U'$. 

Our next lemma yields an estimate on the derivative of the maps in $\Phi_\Gamma$ analogous to the one in part (i) of \cite[Lemma~3.3]{solomyak}.

\begin{lma} \label{key bound}
There exists $C>0$ and $N \in \N$ such that for all $n \geq N$, $A \in \Gamma^n$ and $x \in U$,
\[
\norm{A}^{-2} \leq |\phi_{A}'(x)| \leq C \norm{A}^{-2}.
\]
\end{lma}

\begin{proof}
Since $\overline{U}\cap R_\A=\emptyset$, we can fix $\epsilon>0$ small enough such that if $U_1 \subsetneq U_2$ are such that $U_1$ is an open $\epsilon$-neighbourhood of $U$, and $U_2$ is an open $\epsilon$-neighbourhood of $U_1$, then
\begin{equation}\label{c}
\bigcup_{A\in\Gamma^*}\phi_A\left(\overline{U_i}\right)\subsetneq U_i, \ \text{for}\  i=1,2.
\end{equation}
For the fixed $\epsilon$ used in the definition of $U_1$ and $U_2$, Lemma \ref{pi lemma} implies that there exists $C_\epsilon>1$ and $N$ be large enough, such that
\begin{equation}\label{pi thing}
\pi - \frac{C_\epsilon}{\norm{B}^{2}}>\lvert U_2\rvert,\ \text{for all}\ B\in \Gamma^N,
\end{equation}
where $\lvert U_2 \rvert$ denotes the Lebesgue measure of $U_2$ in $\R\P^1$. Suppose that for some $n \geq N$, there exists $B\in \Gamma^n$ such that $u^-_{B} \in U_1$, that is, $(u^-_{B}-\epsilon, u^-_{B}+\epsilon) \subseteq U_2$. By Lemma \ref{pi lemma},
\[
\lvert\phi_{B}(U_2)\rvert>\pi- \frac{C_\epsilon}{\norm{B}^2}>\lvert U_2\rvert,
\]
which is a contradiction by the invariance of $U_2$. Hence for all $B \in \bigcup_{n \geq N} \Gamma^n$, the points $u^-_{B}$ do not lie in $U_1$, and therefore $(u^-_{B}-\epsilon, u^-_{B}+\epsilon)$ does not intersect $U$. The desired inequality now follows by applying Lemma \ref{contract} for the point $x \in U$ and letting $C=C_\epsilon$.
\end{proof}

Using Lemma \ref{key bound} we can obtain the desired upper bound.

\begin{prop}\label{ub}
Suppose $\A$ is a finite, irreducible and semidiscrete subset of $\SL$. Then
\[
\dim K_\A \leq \min\{1,\delta_\A\}.
\]
\end{prop}

\begin{proof}
Let $\delta_\Gamma$ be the critical exponent of the zeta function for $\Gamma$ and note that, as $\Gamma^*\subseteq \A^*$, we have that $\delta_\Gamma\leq \delta_\A$. Take $\epsilon>0$, and let $N$ be given by Lemma \ref{key bound}. Since $K_\Gamma\subseteq U'$, where $K_\Gamma$ denotes the attractor of $\Gamma$, and since $\Gamma$ is uniformly hyperbolic, we can take $n$ sufficiently large so that $\{\phi_{B}(U)\}_{B\in \Gamma^n}$ is a $\epsilon$-cover of $K_\Gamma$. Then, for all $s>\delta_\Gamma$, we have that
\[
\sum_{B\in \Gamma^n} \lvert\phi_{B} U\rvert^s \leq C^s \sum_{B\in \Gamma^n} \norm{B}^{-2s},
\]
by Lemma \ref{key bound}. In particular the $s$-dimensional Hausdorff measure of $K_\Gamma$, denoted by $H^s(K_\Gamma)$ satisfies
\[
H^s(K_\Gamma) \leq \lim_{n \to \infty} C^s \sum_{B \in \Gamma^n} \norm{B}^{-2s} \leq C^s \zeta_\A(s)< \infty.
\]
So $\hd(K_\Gamma) \leq s$, but since $s>\delta_\Gamma$ was chosen arbitrarily, $\hd K_\Gamma \leq \delta_\Gamma\leq \delta_\A$.\\
In order to conclude the proof of the upper bound, it suffices to show that $\hd(K_\Gamma)=\hd(K_\A)$. Because $K_\Gamma \subseteq K_\A$, we have that $\hd(K_\Gamma)\leq \hd(K_\A)$. For the other side of this inequality, observe that $\hd(\phi_{A_0}(K_\A))=\hd(K_\A)$. Our proof will thus be complete upon showing that $\phi_{A_0}(K_\A)\subseteq K_\Gamma$. Note that since the map $\psi$ is a bijection, by Theorem \ref{attlimset} this is equivalent to showing that $f_{A_0}(\att(F_\A))\subseteq \att(F_\Gamma)$.\\
Let $x\in\att(F_\A)$. By definition, there exists a sequence $(A_n)\subseteq \A^*$, such that $(f_{A_n}(i))$ converges to $x$. Hence, $f_{A_0}f_{A_n}(i)=f_{A_0A_n}(i)$ converges to $f_{A_0}(x)$, as $n\to \infty$. But, $A_0A_n$ lies in $\Gamma$, for all $n$, and so $f_{A_0}(x)\in \att(F_\Gamma)$, as required.
\end{proof}

\section{Continuity of the dimension}\label{cty_sect}

In this section we will prove Theorem \ref{cty}. The proof will follow in two steps. First we will prove a continuity result for the critical exponent, Theorem \ref{cty_str} which will settle the cases where the limit point $\A$ in the statement of Theorem \ref{cty} is either (a) a finite, semidiscrete, irreducible set $\A$ or (b) a finite, uniformly hyperbolic, reducible set $\A$. Then we will separately consider the continuity of the dimension in the case that the limit point $\A$ is reducible and semidiscrete but not uniformly hyperbolic.

Let us first consider the action of $\SL$ on $\R^2$. A \emph{cone} $C$ in $\R^2$ is defined to be a set
\[
C=\{xv_1+yv_2\colon x,y>0\},
\]
for some linearly independent vectors $v_1,v_2$ in $\R^2$. A \emph{multicone} in $\R^2$ is a finite collection of cones in $\R^2$, with disjoint closures. Each multicone $M\subsetneq \R\P^1$ corresponds to a multicone $\tilde{M}$ in $\R^2$ and vice versa. Hence, if $A\in\SL$ and $M$ is a multicone in $\R\P^1$ that gets mapped inside itself by $\phi_A$, then there exists a multicone $\tilde{M}$ in $\R^2$ that gets mapped inside itself by $A$.

We now present a standard result which shows that the norm, when defined on a set of matrices which maps a multicone compactly inside itself, inherits a certain multiplicativity property.

\begin{lma} \label{almost}
Let $K', K\subsetneq \R\P^1$ be two multicones with $\overline{K'} \subsetneq K$. There exists $c>0$ that depends only on $K, K'$ such that for any $A, B\in \SL$ with $\phi_A(K)\subseteq K'$ and $\phi_B(K)\subseteq K'$ we have
\[
\norm{AB} \geq c\norm{A}\norm{B}.
\]
\end{lma}

\begin{proof}
If $\tilde{K}$ and $\tilde{K}'$ are multicones in $\R^2$ that correspond to $K$ and $K'$, respectively, then the closure of $\tilde{K}'$ is contained in $\tilde{K}$ and $A\tilde{K}\subseteq \tilde{K}'$. We begin by claiming that there exists $c>0$ such that
\[
\norm{Aw} \geq c\norm{A} \norm{w},
\]
for all $w \in \tilde{K}'$. That is,
\[
c\norm{A} \leq \frac{\norm{Aw}}{\norm{w}} \leq \norm{A}.
\]
Suppose such a constant $c>0$ does not exist. Then we can construct $A_n$ of norm 1 with $A_n \tilde{K} \subseteq \tilde{K}'$, and $w_n \in \tilde{K}'$ of norm 1 such that $\norm{A_nw_n}<1/n$. Thus looking along a subsequence we can find $A$ of norm 1 such that $A\tilde{K} \subseteq \tilde{K}'$ and $w \in \tilde{K}'$ such that $Aw=0$. 

Next choose $u \in \tilde{K}$ such that $Au \neq 0$ and $w-u \in \tilde{K}$, which is possible because $w \in \overline{\tilde{K}'} \subsetneq \tilde{K}$. In particular $A(w-u) \in \tilde{K}'$ since $A\tilde{K} \subseteq \tilde{K}'$ but also $A(w-u)=-Au \in -\tilde{K}'$, which is a contradiction since $\tilde{K}$ is a multicone.\\
Now, to prove the lemma suppose that $A, B$ map $\tilde{K}$ into $\tilde{K}'$ and let $w \in \tilde{K}'$ with norm 1. Then
\[
\norm{AB} \geq \frac{\norm{ABw}}{\norm{w}} \geq c\norm{A}\norm{Bw} \geq c^2 \norm{A}\norm{B}\norm{w}=c^2\norm{A}\norm{B}.\qedhere
\]
\end{proof}

Lemma \ref{almost} will enable us to prove the aforementioned continuity theorem for the critical exponent, which is our first step towards proving Theorem \ref{cty}. Recall that for a set $\mathcal{B}\subseteq \SL$, we denote by $\delta_\mathcal{B}$ the critical exponent of the zeta function of $\mathcal{B}$, and by $s_\mathcal{B}$ the root of the pressure function of $\mathcal{B}$.

\begin{thm}\label{cty_str}
Let $(\A_n)$ be a sequence of finite, semidiscrete subsets of $\SL$ that converges in the Hausdorff metric to a finite set $\A\subseteq \SL$. Suppose that $\A$ is either:
\begin{enumerate}
\item semidiscrete and irreducible; or
\item uniformly hyperbolic and reducible.
\end{enumerate}
Then,
\begin{align*}
&\delta_{\A_n}\xrightarrow{n\to\infty}\delta_\A &\text{and}& &s_{\A_n}\xrightarrow{n\to\infty} s_\A.
\end{align*}
\end{thm}

\begin{proof} Note that due to Theorem~\ref{s=d} and \cite[Theorem 2]{deleo1}, $s_{\A_n}=\delta_{\A_n}$ and $s_\A=\delta_\A$. So it suffices to prove the result for $s_{\A_n}$ and $s_\A$. We begin by assuming that $\A$ is finite, semidiscrete and irreducible. Thus Proposition \ref{uh-sub} is applicable and yields a matrix $A_0 \in \A^{n_0}$, for some $n_0\in\N$, and open intervals $U, U' \subsetneq \R\P^1$ with $\overline{U'} \subsetneq U$ and $\overline{U}\cap R_\A=\emptyset$, such that $\phi_{A_0}(\overline{U}) \subsetneq U'$. Without loss of generality we can assume that $n_0=1$ and so $A_0\in\A$. In particular, if we write $\A=\{A_1,A_2,\dots, A_M\}$, we assume that $A_0=A_1$.\\
Recall the set 
\[
\A_\infty=\{A_1A\colon A\in (\A\setminus\{A_1\})^*\}\cup\{A_1\},
\]
which was used in the proof of Theorem~\ref{s=d}. $\A_\infty$ is uniformly hyperbolic and satisfies $s_{\A_\infty}=s_\A$. By choosing $n$ large enough we may assume that $\A$ and $\A_n$ have the same cardinality, for all $n$. So if we write $\A_n=\{A_{1,n}, A_{2,n}, \dots A_{M,n}\}$, then $A_{i,n}$ converges to $A_i$, as $n\to\infty$, for $i=1,2,\dots,M$. Thus, we can also define
\[
\A_{\infty,n}=\{A_{1,n}A\colon A\in (\A_n\setminus\{A_{1,n}\})^*\}\cup\{A_{1,n}\}.
\]
We first claim that $\A_{\infty,n}$ is uniformly hyperbolic, for all $n$ large enough. Let $R_n$ be the repeller of $\A_n$. If there exists a subsequence $(\A_{n_k})$ of $(\A_n)$ such that $\overline{U}\cap R_n\neq \emptyset$ for all $n$, then since the repeller is the closure of the repelling fixed points of hyperbolic transformations in $\Phi_{\A_n}$ and unique fixed points of parabolic transformations in $\Phi_{\A_n}$, the convergence of $\A_n$ to $\A$ implies that $\overline{U}\cap R_\A\neq \emptyset$, which is a contradiction.\\
Thus, $\overline{U}\cap R_n=\emptyset$ for all $n$ large enough. So, for a fixed $\epsilon>0$ small enough, there exists $N\in \N$, such that $\overline{\phi_A(U)}\subsetneq U_\epsilon'$, for all $A\in \A_n^*$ and all $n\geq N$, where $U_\epsilon'$ is an open $\epsilon$-neighbourhood of $U'$ in $\R\P^1$, compactly contained in $U$. Applying Theorem~\ref{abybv} concludes the proof of our claim.\\
Working similarly as we did in the proof of Theorem~\ref{s=d}, we can show that $s_{\A_{\infty,n}}=s_{\A_n}$. Hence it suffices to prove that $s_{\A_{\infty,n}}\xrightarrow{n\to\infty}s_{\A_\infty}$.\\
We prove that for all $s>0$,
\begin{equation} \label{limit}
\lim_{n \to \infty}P_{\A_{\infty,n}}(s) = P_{\A_\infty}(s).
\end{equation}
Let us first define the sets
\[
\A_{k,n}\vcentcolon= \A_{\infty,n}\bigcap \left(\bigcup_{i=1}^k\A_n^i\right),
\]
for all $n\in\N$. That is, $\A_{k,n}$ is the set of all words in $\A_{\infty,n}$ that have length at most $k$ in $\A_n$. Similarly, for any $k\in\N$ we define
\[
\mathcal{B}_k=\A_\infty\bigcap \left(\bigcup_{i=1}^k\A^i\right),
\]
which is the the set of all words in $\A_\infty$ that have length at most $k$ in $\A$. In particular $\A_{k,n} \xrightarrow{n\to\infty} \mathcal{B}_k$ in the Hausdorff metric. In the proof of Theorem~\ref{s=d} we proved that
\begin{align*}
&\lim_{k\to\infty} P_{\mathcal{B}_k}(s)=P_{\A_\infty}(s) &\text{and}& &\lim_{k\to\infty} P_{\A_{k,n}}(s)=P_{\A_{\infty,n}}(s).
\end{align*}
So, in order to prove \eqref{limit}, it suffices to prove that 
\begin{equation}\label{fin}
\lim_{n\to\infty}P_{\A_{k,n}}(s)=P_{\mathcal{B}_k}(s),\quad \text{for all} \quad k\in\N.
\end{equation}
Fix $k\in\N$ and note that for each $n$ there exists a set of words $\I_{k,n}$ in the alphabet $\{1,2,\dots,M\}$, such that $\A_{k,n}=\{A_\i\colon \i\in \I_{k,n}\}$. The set $\I_{k,n}$ has to be finite, and so the function
\[
\mathcal{C} \mapsto \left(\sum_{C\in\mathcal{C}_{k,n}^l}\norm{C}^{-2s}\right)^\frac{1}{l},
\]
where $\mathcal{C}_{k,n}=\{C_\i\in\mathcal{C}^*\colon \i\in \I_{k,n}\}$, is continuous in the Hausdorff topology of $\SL$. Thus, if we define
\[
Z_l(s,\mathcal{C}_{k,n})\vcentcolon =\sum_{C\in\mathcal{C}_{k,n}^l}\norm{C}^{-2s},
\]
the functions $\mathcal{C}\mapsto \sup_l\left(Z_l(s,\mathcal{C}_{k,n})\right)^\frac{1}{l}$ are lower semicontinuous, for every $n$. Hence, we have that
\[
\liminf_{n \to \infty} P_{\A_{k,n}}(s)=\liminf_{n \to \infty} \sup_l \left(Z_l(s,\A_{k,n})\right)^\frac{1}{l} \geq \sup_l \left(\sum_{B\in\mathcal{B}_k^l}\norm{B}^{-2s}\right)^\frac{1}{l}=P_{\mathcal{B}_k}(s).
\]
So it remains to show that $\limsup_{n \to \infty}P_{\A_{k,n}}(s) \leq P_{\mathcal{B}_k}(s)$. Due to our claim, Lemma \ref{almost} is applicable and yields a constant $c>0$ such that for any $A, B \in \bigcup_{n \geq N} \A_{\infty,n}^*$,
\begin{equation} \label{am}
\norm{AB} \geq c\norm{A} \norm{B}.
\end{equation}
It follows that for any $n\geq N$, and all $l,m\in\N$
\[
Z_{l+m}(s, \A_{k,n}) \leq  c^{-2s}Z_l(s, \A_{k,n})Z_m(s, \A_{k,n}).
\]
Therefore, $P_{\A_{k,n}}(s)=\lim_{l \to \infty} \left(c^{\frac{-2s}{3}}Z_l(s, \A_{k,n})\right)^\frac{1}{l}= \inf_l \left(c^{\frac{-2s}{3}}Z_l(s, \A_{k,n})\right)^{\frac{1}{l}}$ by Fekete's theorem for submultiplicative sequences. Inequality \eqref{am} also holds for any $A,B\in\A_\infty$, and so
\[
P_{\mathcal{B}_k}(s)=\inf_l \left(\sum_{B\in\mathcal{B}_k^l}\norm{A}^{-2s}\right)^\frac{1}{l}.
\]
Thus we obtain that
\[
\limsup_{n \to \infty} P_{\A_{k,n}}(s)=\limsup_{n \to \infty} \inf_l \left(c^{\frac{-2s}{3}}Z_l(s, \A_{k,n})\right)^{\frac{1}{l}} \leq \inf_l \left(\sum_{B\in\mathcal{B}_k^l}\norm{B}^{-2s}\right)^\frac{1}{l}=P_{\mathcal{B}}(s),
\]
since $\mathcal{C} \mapsto \inf_l \left(Z_l(s, \mathcal{C}_{k,n})\right)^\frac{1}{l}$ are upper semicontinuous functions for any $n$. This proves \eqref{fin} and so equation \eqref{limit} is proved.\\
To conclude the proof of the theorem, take any $\epsilon>0$. Due to \eqref{limit} we can choose $N$ sufficiently large that for all $n \geq N$, 
\begin{align*}
&P_{\A_{\infty,n}}(s_{\A_\infty}-\epsilon)>1 &\text{and}& &P_{\A_{\infty,n}}(s_{\A_\infty}+\epsilon)<1,
\end{align*}
implying that 
\[
s_{\A_\infty}-\epsilon<s_{\A_{\infty,n}}<s_{\A_\infty}+\epsilon,
\] 
as required.

Next, we assume that $\A$ is uniformly hyperbolic and reducible. In this case the proof is just a simplified version of the arguments that we used above. The fact that 
\[
\liminf_{n \to \infty} P_{\A_n}(s) \geq P_\A(s),
\]
follows from the submultiplicativity of the matrix norm, as above. On the other hand, by uniform hyperbolicity and reducibility of $\A$ there exist cones $K, K' \subsetneq \R\P^1$ with $\overline{K'} \subsetneq K$ such that $\Phi_\A(K) \subseteq K'$. By Lemma \ref{almost} there exists $c>0$ which depends only on $K, K'$ such that for all $A, B \in \SL$ with $\phi_A(K) \subseteq K'$ and $\phi_B(K) \subseteq K'$, 
\[
\norm{AB} \geq c \norm{A}\norm{B}.
\]
Using this we can follow an analogous argument to the above to deduce that 
\[
\limsup_{n \to \infty} P_{\A_n}(s) \leq P_\A(s).
\]
In particular, $\lim_{n \to \infty}s_{\A_n} = s_\A$, completing the proof.
\end{proof}

In view of Theorem \ref{cty_str}, in order to prove Theorem \ref{cty} it remains to consider the case where the limit point $\A$ is reducible and semidiscrete but not uniformly hyperbolic and $K_\A$ is not a singleton. Then, as mentioned in Section~\ref{eg}, $\A$ must necessarily fall into one of the two following categories:
\begin{enumerate}
\item either there exists a hyperbolic matrix $A$ and a parabolic matrix $B$ in $\A$, such that $\phi_B$ fixes the repelling fixed point of $\phi_A$; 
\item or there exist two hyperbolic matrices $A$ and $B$ in $\A$, such that the attracting fixed point of $\phi_A$ is the repelling fixed point of $\phi_B$.
\end{enumerate}

In Section~\ref{eg} we showed that in both of the above categories $K_\A$ contains an interval, and so $\hd K_\A=1$. We are going to show that $K_{\A_n}$ contains an interval for all $n$ large enough, thus concluding the proof of Theorem~\ref{cty}. We focus on category (1); the proof for category (2) is similar.\\

Suppose that $A_n,B_n\in \A_n$ are such that $A_n\to A$ and $B_n\to B$, as $n\to\infty$. Note that the matrices $B_n$ are either hyperbolic or parabolic, while $A_n$ have to be hyperbolic for $n$ large enough. We assume that $B_n$ is parabolic for all $n$; our arguments can be easily modified to work for the case where $B_n$ is hyperbolic.

Let $a(A_n)$ and $r(A_n)$ be the attracting and the repelling fixed points of $\phi_{A_n}$, respectively. Similarly, let $p(B_n)$ be the unique fixed point of $\phi_{B_n}$. If $p(B_n)=r(A_n)$ for all $n$ large enough, then the attractor of $\{\phi_{A_n},\phi_{B_n}\}$ is an interval and we have nothing to prove. So, assume that $p(B_n)$ and $r(A_n)$ are distinct along a subsequence. By relabelling this subsequence, we assume that $p(B_n)\neq r(A_n)$ for all $n$.

Let us recall the following result from \cite[Theorem~5.23]{thesis}.
\begin{thm}\label{principal}
Let $(\A_n)$ be a sequence of semidiscrete subsets of $\SL$ that converges to a finite and semidiscrete set $\A$. If $\Phi_\A$ maps an open interval of $\R\P^1$ strictly inside itself, then for all $n$ large enough $\Phi_{\A_n}$ also maps an open interval strictly inside itself.
\end{thm}
Theorem~\ref{principal} implies that, for all $n$ large enough, the projective IFS $\{\phi_{A_n},\phi_{B_n}\}$ maps an open interval $I_n\subsetneq \R\P^1$ strictly inside itself. By conjugating we can assume that $\overline{I_n}\subsetneq (0,\pi)$. Then $a(A_n)\in I_n$ and $p(B_n)\in\partial I_n$, whereas $r(A_n)$ lies in the complement of $\overline{I_n}$. So $I_n$ can be chosen to be the interval with endpoints $a(A_n)$ and $p(B_n)$ that does not contain $r(A_n)$. Without loss of generality, we assume that $a(A_n)<p(B_n)<r(A_n)$ for all $n$ large enough, i.e. $I_n=(a(A_n),p(B_n))$.

We are going to show that $K_{\A_n}$ contains $I_n$ for all $n$ large enough, which yields the desired result. It suffices to show that the attractor of $\{\phi_{A_n},\phi_{B_n}\}$ contains $I_n$ for all $n$ large enough.

Observe that $a(A_n),r(A_n)$ converge to the attracting and repelling fixed points of $A$, respectively, while $p(B_n)$ converges to the unique fixed point of $B$, which, by assumption, coincides with the repelling fixed point of $A$. This along with the fact that $\norm{A_n}$ and $\norm{B_n}$ are bounded, yields that $\phi_{B_n}(a(A_n))<\phi_{A_n}(p(B_n))$, for all $n$ large enough. Hence $\phi_{A_n}(I_n)\cup\phi_{B_n}(I_n)=I_n$, for all $n$ large enough. Fix such an $n$. So for every $x\in I_n$ there exists $C\in\{A_n,B_n\}$ such that $x\in\phi_C(I_n)$ and we can then recursively find a sequence $(C_k)\subseteq \{A_n,B_n\}$ so that $x\in \phi_{C_1C_2\cdots C_k}(I_n)$, for all $k\in\N$. It is now easy to check that the intervals $\phi_{C_1C_2\cdots C_k}(I_n)$ are nested and their diameter goes to 0 as $k\to\infty$. Therefore, every $x\in I_n$ is an accumulation point of either $a(A_n)$ or $p(B_n)$ under the semigroup generated by $\{\phi_{A_n},\phi_{B_n}\}$, which by the second part of Theorem~\ref{dense} yields that every $x\in I_n$ lies in the attractor of $\{\phi_{A_n},\phi_{B_n}\}$. This completes the proof of Theorem \ref{cty}.

\begin{rem}\label{discont}
It is natural to ask whether we could have a continuity theorem more general than Theorem \ref{cty}. However, the assumptions on the limit $\A$ of the sequence $\A_n$ in the statement of Theorem \ref{cty} cannot be significantly weakened due to the following:

It easy to show that the Hausdorff dimension can be discontinuous when the limit $\A$ generates a non-free semigroup or $K_\A$ is a singleton. Let us show that it is also necessary to assume that $\A$ is semidiscrete. Let $\A_n=\{A_n,B_n,C\}$ be irreducible and Diophantine subsets of $\SL$ such that $\Phi_{\A_n}$ maps an open interval $I$ compactly inside itself, for all $n$. Assume that $\lim_{n\to\infty} A_n=\lim_{n\to\infty} B_n=\mathrm{Id}$. Then $\A_n$ converges to $\A=\{\mathrm{Id}, \mathrm{Id}, C\}$, which is not semidiscrete and $\hd K_\A=0$. Then, $\phi_{A_n}(I)\cup \phi_{B_n}(I)=I$ for all $n$ large enough, and so $\hd K_{\A_n}=1$ for all $n$ large enough, which implies that the dimension is discontinuous at $\A$.
\end{rem}

\noindent \textbf{Acknowledgements.}  Both authors were financially supported by the \emph{Leverhulme Trust} (Research Project Grant number RPG-2016-194). The second author was also financially supported by the \emph{EPSRC} (Standard Grant EP/R015104/1).

\end{document}